\documentclass[11pt]{amsart}

\usepackage{amsfonts,amssymb,amsmath}
\usepackage{graphicx}
\usepackage{psfrag}
\usepackage{color}
\usepackage{verbatim}
\usepackage{epstopdf}

\usepackage[text={6in,9in},centering]{geometry}

\newtheorem{thm}{Theorem}[section]
\newtheorem{lem}[thm]{Lemma}

\newtheorem{cor}[thm]{Corollary}

\theoremstyle{definition}

\theoremstyle{remark}

\numberwithin{equation}{section}

\newcommand{\bR}{{\mathbb R}}

\newcommand{\bZp}{{{\mathbb Z}^+}}

\newcommand{\Z}{{\mathbb Z}}

\newcommand{\bfe}{{\mathbf e}}

\newcommand{\bfw}{{\mathbf w}}
\newcommand{\bfx}{{\mathbf x}}
\newcommand{\bfu}{{\mathbf u}}
\newcommand{\bfv}{{\mathbf v}}

\newcommand{\mA}{{\mathcal A}}
\newcommand{\mB}{{\mathcal B}}

\newcommand{\upM}{{\overline{M}}}
\newcommand{\lowM}{{\underline{M}}}

\newcommand{\uplam}{{\overline{\lambda}}}

\newcommand{\upd}{{\overline{d}}}
\newcommand{\lowd}{{\underline{d}}}

\newcommand{\ups}{{\overline{s} }}
\newcommand{\lows}{{\underline{s}}}

\newcommand{\uprho}{{{\rho}}}

\newcommand{\Var}{\mathrm{Var}}

\def\vep{\varepsilon}

\def\R{\mathbb{R}}

\def\wdt{\widetilde}

\def\updimB{\overline{\dim}_B}
\def\lowdimB{\underline{\dim}_B}

\newcommand{\card}{{\rm card\,}}

\begin{document}


\title[Box dimension of RFIFs ]{Box dimension of the graphs of \\ recurrent fractal interpolation functions}


\author{Lai Jiang}
\address{School of Fundamental Physics and Mathematical Sciences, Hangzhou Institute for Advanced Study, University of Chinese Academy of Sciences, Hangzhou 310024, China}
\email{jianglai@ucas.ac.cn}

\author{Xiao-Hui Li}
\address{School of Mathematics and Information Sciences, Yantai University, Yantai 264000, China}
\email{lxhmath@ytu.edu.cn}


\author{Zhen Liang}
\address{School of Mathematics, Nanjing Audit University, Nanjing 211815, China}
\email{zhenliang@nau.edu.cn}

\author{Huo-Jun Ruan}
\address{School of Mathematical Sciences, Zhejiang University, Hangzhou 310058, China}
\email{ruanhj@zju.edu.cn}

\thanks{Corresponding author: Lai Jiang}

\date{}


\begin{abstract}
Let $f$ be a generalized affine recurrent fractal interpolation function with vertical scaling functions. In this paper, by introducing underlying local iterated function systems of $f$, we define restricted vertical scaling matrices. Then we prove the monotonicity of spectral radii of these matrices without additional conditions. We also prove the irreducibility of these matrices under the assumption that vertical scaling functions are positive. With these results, we estimate the upper and lower box dimensions of the graphs of $f$ by the limits of spectral radii of restricted vertical scaling matrices. In particular, we obtain an explicit formula of the box dimension of the graph of $f$ under certain constraint conditions.
\end{abstract}

\maketitle

\section{Introduction}

Fractal functions, including Weierstrass functions and Takagi functions, are known as continuous but nowhere differentiable functions. There are numerous works on these functions. One of interests focuses on fractal dimensions of the graphs of these functions, including Hausdorff dimension and box dimension, see \cite{AllKaw11, BBR14, HuLau93, RenShen21, Shen18} and the references therein.

The classical Weierstrass function is defined by
\[
   W_{\lambda,b}(x)=\sum_{n=0}^\infty \lambda^n \cos(2\pi b^n x), \quad x\in [0,1],
\]  
where $b\geq 2$ is an integer and $1/b<\lambda<1$. It is well known that the graphs of these functions are invariant sets of certain iterated function systems (IFSs), see \cite[section~6]{JR24} for example. 

In 1986, Baransley \cite{B86} introduced fractal interpolation functions (FIFs). Basically, these functions interpolate give data while the graphs of these functions are invariant sets of certain IFSs. As a result, the classical Weierstrass functions and the classical Takagi function are all FIFs. Since 1986, FIFs have attracted many attentions to study their analytic properties \cite{All20,BH89,CKV15,Dubuc18,WY13} and fractal dimensions \cite{APN17,BRS20,BaMa15,HarMas86,JhaVer21,RSY09}. Meanwhile, these functions are suitable to fit fractal-like data, including speech data and the profiles of mountain ranges \cite{MH92,VDL94}.

Notice that the graph of a FIF is an invariant set of certain IFS. Thus, roughly speaking, a small piece of the graph is a copy of the whole graph. In order to generalize FIFs, Barnsley, Elton and Hardin \cite{BEH89} used recurrent IFSs to construct recurrent FIFs as follows.

Let $N\geq 2$ and $\{(x_n, y_n)\}_{n=0}^N $ be a given data set, where $x_0 < x_1 < \cdots < x_N$. Fix the index pairs $ \{\ell(n), r(n)\}_{n=1}^N \subset \{0,1,\ldots,N\} $ with $ \ell(n) < r(n) $ for all $n$. For each $1 \leq n \leq N $, we define $I_n = [x_{n-1}, x_n] $ and $ D_n = [x_{\ell(n)}, x_{r(n)}] $, and assume that the following conditions hold.

\begin{enumerate}
	\item $ L_n: D_n \to I_n $ is a contractive homeomorphism.
	\item $ F_n: D_n \times \mathbb{R} \to \mathbb{R} $ is continuous, and there exists $ 0 < \beta_n < 1 $ such that
	$$|F_n(x, y') - F_n(x, y'')| \leq \beta_n |y' - y''|, \quad x \in D_n, \;\; y', y'' \in \mathbb{R}.$$
	\item If we define $ W_{n}: D_n \times \mathbb{R} \to I_n \times \mathbb{R} $ by
	$W_{n}(x, y) = ( L_n(x), F_n(x, y))$, then
   $W_n$ satisfies the following interpolation constraint: 
	\[ 
	W_{n}\big(\{(x_{\ell(n)}, y_{\ell(n)}), (x_{r(n)}, y_{r(n)}) \}\big)= \{(x_{n-1}, y_{n-1}), (x_n, y_n)\}.
	\]
\end{enumerate}

Given a function $g$ on $E \subset\bR$, we denote the graph of $g$ by $\Gamma g=\{(x, g(x)):\, x \in E\}$. Additionally, assume that $J$ is a subset of $E$, we write $g|_J$ the restriction of $g$ on $J$, and let $\Gamma g|_{J}$ be the graph of $g|_{J}$.

\begin{thm}[\cite{BEH89,RXY21}]\label{thm:existence}
	There exists a unique continuous function $f$ on $[x_0, x_N]$ such that $ f(x_n) = y_n $ for all $0 \leq n \leq N $, and 
	\begin{equation}\label{eq:rfif}
		\Gamma f|_{I_n} = W_{n} \big( \Gamma f|_{D_n} \big), \quad  1 \leq n \leq N.
	\end{equation}
\end{thm}

The function $f$ in this theorem is called a \emph{recurrent fractal interpolation function} (RFIF) determined by the recurrent IFS $\{W_1,\ldots,W_N\}$. In the original work \cite{BEH89}, all $W_n$ are affine maps. The corresponding RFIF is called an \emph{affine RFIF}. Barnsley et. al. \cite{BEH89} obtained the box dimension of the graphs of affine RFIFs under certain conditions. This result was generalized by Ruan, Xiao and Yang \cite{RXY21}. We remark that in \cite{RXY21}, $L_n$ can be not contractive.

We call the function $f$ in Theorem~\ref{thm:existence} a \emph{generalized affine RFIF} if for each $1\leq n\leq N$,
\[
  L_n(x)=a_n x+b_n, \quad F_n(x,y)= S_n(x)y + q_n(x),
\]
where $a_n$ and $b_n$ are real numbers, $S_n(x)$ and $q_n(x)$ are continuous function on $D_n$ with $|S_n(x)|<1$ for all $x\in D_n$. 
It is very challenging to calculate the box dimension of the graphs of generalized affine RFIFs without any restriction. In this paper, we will study $\dim_B \Gamma f$, the box dimension of $\Gamma f$, where the following conditions are satisfied for each $n$.
\begin{itemize}
	\item[(A1)] $x_n-x_{n-1}=(x_N-x_0)/N$.
	\item[(A2)] $S_n$ is a continuous function of bounded variation on $D_n$ and $|S_n(x)|<1$ for all $x\in D_n$.
	\item[(A3)] $q_n$ is a continuous function of bounded variation on $D_n$.
\end{itemize}

In \cite{JR23, JR24}, the first author and the last author introduced the sequences of upper and lower vertical scaling matrices, and used the limits of spectral radii of these matrices to estimate the box dimension of the graphs of generalized affine FIFs. 
Naturally, we want to use similar method in the recurrent case. Notice that in \cite{JR23,JR24}, we have the following key property: lower vertical scaling matrices are irreducible if vertical scaling functions are positive. However, this key property does not hold in the recurrent case. Thus we need new method to estimate the box dimension of $\Gamma f$ in the present setting.

Now we sketch our strategy. Firstly, by using the relationship between the intervals $\{I_n;D_n;n=1,2\ldots,N\}$, we define strongly connected components $\Lambda_1,\ldots,\Lambda_m$ of $\{1,\ldots,N\}$. Then we introduce underlying local IFSs $\{L_n:\, n\in \Lambda_r\}$ of the generalized affine RFIF $f$. By using these local IFSs, we define restricted vertical scaling matrices. With these matrices, we can estimate the box dimension of $\Gamma f$ on a sequence of basic sets related with local IFS $\{L_n:\, n\in \Lambda_r\}$. On the other hand, we directly estimate the box dimension of $\Gamma f$ on the complement of basic sets. 
Finally, we successfully estimate the box dimension of the graphs of generalized affine RFIFs. See Section~2 for the definition of $\Lambda_r$, restricted vertical scaling matrices, and basic sets.

 The paper is organized as follows. 
 In Section 2, we give some definitions including restricted vertical scaling matrices, and also present main results. 
 In Section 3, we study irreducibility of restricted vertical scaling matrices and the monotonicity of spectral radii of these matrices. In Section 4, we present some inequalities and equalities for the box dimension of the graphs of generalized affine RFIFs on various sub-intervals of $[x_0,x_N]$. By using these results, in Section 5, we estimate the box dimension of the graphs of generalized affine RFIFs by the limits of spectral radii of restricted vertical scaling matrices. In Section~6, we present an example to illustrate our results.

\section{Preliminaries and main results}

\subsection{Basic sets and related notations}

For any $i,j\in \{1,\ldots,N\}$, we say there is an \emph{edge} from $j$ to $i$ if $I_j \subset D_i$.
Let $V$ be a nonempty subset of $\{1,\ldots,N\}$. For any $i,j\in V$, a finite sequence $\{i_{\ell}\}_{\ell=0}^{t}$ is called a \emph{path} in $V$ from $j$ to $i$, if $i_{0}=j$, $i_t=i$, $i_{\ell}\in V$ for all $0\leq \ell\leq t$, and there is an edge from $i_{\ell-1}$ to $i_{\ell}$ for each $1\leq \ell\leq t$. The set $V$ is called \emph{strongly connected} if for any $i,j\in V$, there is a path in $V$ from $j$ to $i$. Furthermore, $V$ is called a \emph{strongly connected component} of $\{1,\ldots,N\}$, if it is strongly connected and there is no strongly connected subset $U$ of $\{1,\ldots,N\}$ such that $V\subsetneq U$.

In the sequel of the paper, we always assume that $\{ \Lambda_1,\ldots,\Lambda_m\}$ is the family of all strongly connected components of $\{1,\ldots,N\}$. Given $r=1,\ldots,m$,
we naturally have a local IFS $\{D_n;L_n: n\in \Lambda_r\}$ (or $\{L_n: n\in \Lambda_r\}$ for short), while the corresponding local Hutchinson-Barnsley operator $W_{\textrm{local},r}:\, \mathcal{X}_r\to \mathcal{X}_r$ is defined by
\[
  W_{\textrm{local},r}(A)=\bigcup_{n\in \Lambda_r} L_n(A\cap D_n), \quad A\in \mathcal{X}_r.
\]
Here $\mathcal{X}_r$ is the family of all subsets of $\bigcup_{n\in \Lambda_r}D_n$. A nonempty set $K$ in $\mathcal{X}_r$ is called an invariant set of the local IFS $\{L_n: n\in \Lambda_r\}$ if $W_{\textrm{local},r}(K)=K$. We call $\{L_n:\, n\in \Lambda_r\}$, $1\leq r\leq m$, underlying local IFSs of the generalized affine RFIF $f$. For more details on local IFS, please see \cite{BarHur93,Kou08} and the references therein. 

For $r=1,\ldots,m$, we define $B_{r,0}=\bigcup_{n\in \Lambda_r} D_n$ and 
\begin{equation}\label{eq:B-rk-RecRelation}
	B_{r,k}=W_{\textrm{local},r}(B_{r,k-1})=\bigcup_{n\in \Lambda_r} L_n(B_{r,k-1}\cap D_n), \quad k \geq 1.
\end{equation}
The set $B_{r,k}$ is called a \emph{$k$-th basic set} of $\{L_n:\, n\in \Lambda_r\}$.
Notice that $B_{r,1}=\bigcup_{n\in \Lambda_r} I_n\subset B_{r,0}$ since $\Lambda_r$ is a strongly connected components of $\{1,\ldots,N\}$. Thus $\{B_{r,k}\}_{k\geq 0}$ is a decreasing sequence of nonempty compact sets with respect to $k$. As a result, $K_r=\bigcap_{k=0}^\infty B_{r,k}$ is a nonempty compact set. It is routine to check that $K_r$ is the maximal invariant set of the local IFS $\{L_n: n\in \Lambda_r\}$, while we do not need the fact in this paper.

There is another way to obtain basic sets. Let $ 1 \leq r \leq m $.
For each $k\geq 0$, we recursively define $\{\mB_{r,k}\}_{k\geq 0}$ by letting $\mB_{r,0}=\{D_n:\, n\in \Lambda_r\}$ and 
\[
  \mB_{r,k}=\big\{L_n(E):\,  n\in \Lambda_r,  E\in \mB_{r,k-1} \mbox{ with } E\subset D_n \big\}, \quad  k \geq 1.
\]
By definition, $\mB_{r,1}=\{I_n:\, n\in \Lambda_r\}$. Each set $E$ in $\mB_{r,k}$ is called a \emph{ $k$-th level basic interval} of the local IFS $\{L_n:\, n\in \Lambda_r\}$. It is obvious that $B_{r,k}=\cup \mB_{r,k}$ for all $r$ and $k$. 

In this paper, we always assume that the local IFS $\{L_n : n \in \Lambda_r\}$ satisfies the following uniformly contracting condition for all $1 \leq r \leq m$. 
\begin{enumerate}
	\item[(A4)] For any $r=1,\ldots,m$, there exists an integer $T_r\geq 2$ such that
	$$\frac{|D_n|}{|I_n|}=T_r, \quad  n \in \Lambda_r,$$ 
	where we use $|J|$ to denote the length of an interval $J$.
\end{enumerate}
Write the cardinality of $\Lambda_r$ by $d_r$. Assume that $\Lambda_r=\{ a_{r,1},a_{r,2}, \ldots ,a_{r,d_r} \}$ with
$a_{r,1}<a_{r,2}<\cdots< a_{r,d_r}$.
For any $k \geq 1 $, $1 \leq t \leq d_r$, and $1 \leq j \leq {T_r}^{k-1}$, we define 
\begin{equation}\label{eq:Iri-def}
   I_{r,(t-1){T_r}^{k-1}+j}^{k}=\Big[ x_{(a_{r,t}-1)}+\frac{(j-1)(x_N-x_0)}{N{T_r}^{k-1}}, x_{(a_{r,t}-1)}+\frac{j(x_N-x_0)}{N{T_r}^{k-1}}  \Big].
\end{equation}
For example, for each $1\leq t\leq d_r$,
\begin{equation}\label{eq:Irt-One}
	I_{r,t}^1=I_{r,(t-1){T_r}^0+1}^1=I_{a_{r,t}}.
\end{equation}

Notice that for each $E\in \mB_{r,k}$ with $k \geq 1$, there exists unique $1\leq i\leq d_r{T_r}^{k-1}$ such that $E=I_{r,i}^k$. Thus, for any $ k \geq 1$, we define
\[
	\Theta_{r,k}=\big\{ 1 \leq  i \leq d_r {T_r}^{k-1}: \,I_{r,i}^k \in \mB_{r,k}  \big\}.
\]

\subsection{Restricted vertical scaling matrices}

Given $1\leq r\leq m$ and $k \in \mathbb{Z}^+$, we first define $d_r {T_r}^{k-1} \times d_r {T_r}^{k-1}$ matrices $\upM_{r,k}$ and $\lowM_{r,k}$ as follows.
For any $1 \leq i,j \leq d_r {T_r}^{k-1}$, let $1 \leq n \leq N$ be the unique element satisfying $I_{r,i}^k \subset I_{n}$. We define 
\[D_{r,i}^k=L_n^{-1}(I_{r,i}^k).\] 
Furthermore, we write
\begin{equation*}
	\ups_{r,i,j}^{k}= 
	\begin{cases}
		\max \big\{ \big| S_{n}(x) \big|:\, x \in   I_{r,j}^k\big\}, \quad & \mbox{if } I_{r,j}^k \subset D_{r,i}^k , \\
		0 ,\quad & \mbox{otherwise},
	\end{cases}
\end{equation*}
and
\begin{equation*}
	\lows_{r,i,j}^{k}=
	\begin{cases}
		\min \big\{ \big| S_{n}(x) \big|:\, x \in   I_{r,j}^k\big\} ,\quad & \mbox{if }  I_{r,j}^k \subset D_{r,i}^k, \\
		0 ,\quad & \mbox{otherwise}.
	\end{cases}
\end{equation*}
With the help of $\ups_{r,i,j}^{k}$ and $\lows_{r,i,j}^{k}$, we define the matrices $\upM_{r,k}$ and $\lowM_{r,k}$ by letting
$$\big(\upM_{r,k} \big)_{i,j} = \ups_{r,i,j}^k, \quad \big(\lowM_{r,k} \big)_{i,j} = \lows_{r,i,j}^k.$$
The matrix $\upM_{r,k}$ (resp. $\lowM_{r,k}$) is called the \emph{level-$k$ upper (resp. lower) vertical scaling matrix} associated with $\Lambda_r$.

Let $A=(a_{ij})_{n \times n}$ be a matrix. For any $ \Theta \subset \{1,2,\ldots,n\}$, we write $A|_\Theta=(a_{ij})_{i,j \in \Theta}$ and call it the \emph{principal submatrix} of $A$. The matrix $\upM_{r,k}|_{\Theta_{r,k}}$ (resp. $\lowM_{r,k}|_{\Theta_{r,k}}$) is called a \emph{upper (resp. lower) restricted vertical scaling matrix} associated with $\Lambda_r$ and $\Theta_{r,k}$. We remark that these matrices play a crucial role in our study.

\subsection{Main results}
First, we have the following results on the irreducibility of restricted vertical scaling matrices.

\begin{thm}\label{thm:irr}
	Let $1\leq r\leq m$. 
	\begin{enumerate}
		\item 	Assume that $S_n$ is positive on $D_n$ for all $n\in\Lambda_r$. Then $\lowM_{r,k}|_{\Theta_{r,k}}$ and $\upM_{r,k}|_{\Theta_{r,k}}$ are irreducible for all $k \in \Z^+$. 
		\item Assume that $S_n$ is positive on $D_n\cap K_r$ for all $n\in\Lambda_r$. Then there exists $k_0 \in \Z^+$ such that $\lowM_{r,k}|_{\Theta_{r,k}}$ and $\upM_{r,k}|_{\Theta_{r,k}}$ are irreducible for all $k \geq k_0$.
	\end{enumerate}
\end{thm}


Given an $n\times n$ matrix $A$, we write
$\rho(A)=\max\{|\lambda_i|: 1 \leq i \leq n\}$, where $\lambda_1,\ldots,\lambda_n$ are eigenvalues of $A$.
We call $\rho(A)$ the \emph{spectral radius} of $A$.
Using the above theorem, we study the monotonicity of spectral radii of $\lowM_{r,k}|_{\Theta_{r,k}}$ and $\upM_{r,k}|_{\Theta_{r,k}}$. 
\begin{thm}\label{thm:rho-up-sub}
	Let $1\leq r\leq m$.
	\begin{enumerate}
		\item 	The sequence $\{\rho(\upM_{r,k}|_{\Theta_{r,k}})\}_{k=1}^\infty$ is decreasing with respect to $k$. Thus, the limit $\lim_{k \to \infty}\rho(\upM_{r,k}|_{\Theta_{r,k}})$ exists, denoted by $\uprho_{r}$.
		\item The sequence $\{\rho(\lowM_{r,k}|_{\Theta_{r,k}})\}_{k=1}^\infty$ is increasing with respect to $k$. Thus, the limit $\lim_{k \to \infty}\rho(\lowM_{r,k}|_{\Theta_{r,k}})$ also exists. Moreover, if $S_n$ is positive on $D_n \cap K_r$ for all $n\in \Lambda_r$, then
		$\lim_{k \to \infty}\rho(\lowM_{r,k}|_{\Theta_{r,k}})=\uprho_r.$		
	\end{enumerate}
\end{thm}


Fix $1\leq r\leq m$. In the case that $S_n$ is positive on $D_n\cap K_r$ for all $n\in \Lambda_r$, let $k_r^*$ be the smallest positive integer such that $S_n$ is positive on $D_n\cap B_{r,k_r^*}$ for all $n\in \Lambda_r$. We define
\begin{equation}\label{eq:def-drstar}
	d_r^*=\begin{cases}
		1+\log\rho_r/\log T_r, & \textrm{if } \Var(f,E)=+\infty \textrm{ for some } E\in \mB_{r,k_r^*}, \\
		1, & \textrm{otherwise}.
	\end{cases}
\end{equation}
Here $\Var(f,E)$ is the classical total variation of $f$ on $E$. 
We obtain the following result on the box dimension of $\Gamma f$. 
\begin{thm}\label{thm:Main-Results-x1}
	Let $f$ be a generalized affine RFIF satisfying conditions (A1)-(A4). Then the following results hold.
	\begin{enumerate}
		\item $\updimB \Gamma f \leq 1 + \max\big\{ \log\uprho_1/ \log T_1,\log\uprho_2/\log T_2,\ldots, \log \uprho_m/ \log T_m,0 \big\}$.
		\item If for all $1\leq r\leq m$ and all $n\in \Lambda_r$, the function $S_n $ is positive on $D_n \cap K_r$, then 
		\[
		 \dim_B \Gamma f = \max\{d_1^*, d_2^*,\ldots, d_m^*,1\}.
		\]		
	\end{enumerate}
\end{thm}


\section{Analysis on restricted vertical scaling matrices}
\subsection{Some facts of basic sets and basic intervals}

By definitions of basic sets and basic intervals, we have the following simple facts. Since their proofs are trivial, we omit the details.

\begin{lem}\label{lem:Theta-simple-facts}
 Let $1\leq r\leq m$ and $k\geq 1$. 
 \begin{enumerate}
	  \item For any two distinct elements $i,j$ in $\{1,2,\ldots,d_r{T_r}^{k-1}\}$, we have $(I_{r,i}^k \cap I_{r,j}^k)^\circ =\emptyset$, where we use $(\cdot)^\circ$ to denote the interior of a set.
	  \item If $i\in\{1,2,\ldots,d_r{T_r}^{k-1}\}\setminus \Theta_{r,k}$, then $(I_{r,i}^k \cap B_{r,k})^\circ = \emptyset$. As a result, $(I_{r,i}^k \cap B_{r,k+1})^\circ = \emptyset$.
	  \item $\Theta_{r,k}=\big\{ 1 \leq  i \leq d_r {T_r}^{k-1}: \,I_{r,i}^k \subset B_{r,k}  \big\}$.	 In particular, $\Theta_{r,1}=\{1,2,\ldots,d_r\}$.
	  \item For any $i\in \Theta_{r,k}$ and $k'>k$, the set $\{i'\in \Theta_{r,k'}:\, I_{r,i'}^{k'}\subset I_{r,i}^k\}$ is nonempty.
	  \item For any $1\leq t\leq d_r$, $D_{r,t}^1=D_{a_{r,t}}$. Furthermore, for any $1\leq i\leq d_r{T_r}^k$, there exist $n_i\in\{1,\ldots,N\}$ and $p_i\in\{1,2,\ldots,{T_r}^{k-1}\}$ dependent of $r$ and $k$, such that
	  \begin{equation}\label{eq:Theta-simple-facts-1}
	    D_{r,i}^{k+1}=[x_{(n_i-1)}+(p_i-1)\varepsilon_{r,k},x_{(n_i-1)}+p_i\varepsilon_{r,k}],
	  \end{equation}
	 where $\varepsilon_{r,k}=(x_N-x_0)N^{-1}{T_r}^{-k+1}$. As a result, we have either $D_{r,i}^{k+1}\in \mB_{r,k}$ or $(D_{r,i}^{k+1}\cap B_{r,k})^\circ=\emptyset$.
  \end{enumerate}	
\end{lem}

By \eqref{eq:Iri-def}, for any $1\leq r\leq m$, $k\geq 1$ and $1\leq i\leq d_r T_r^{k-1}$,
\[
  I_{r,i}^k=\bigcup_{\ell=(i-1)T_r+1}^{iT_r} I_{r,\ell}^{k+1}.
\]
Hence,
$B_{r,k}=\bigcup_{i\in \Theta_{r,k}} I_{r,i}^k=\bigcup_{i\in \Theta_{r,k}} \bigcup_{\ell=(i-1)T_r+1}^{iT_r} I_{r,\ell}^{k+1}$.
We define
\[
  \wdt{\Theta}_{r,k}=\big\{(i-1)T_r+t: \, i\in \Theta_{r,k}, 1\leq t \leq T_r \big\}.
\]
It is clear that
\begin{equation}\label{eq:3-1}
	\wdt{\Theta}_{r,k} =\big\{ 1 \leq  \ell \leq d_r {T_r}^{k}:\, I_{r,\ell}^{k+1} \subset B_{r,k}   \big\}.
\end{equation}

Notice that $B_{r,k+1}\subset B_{r,k}$. Thus, combining \eqref{eq:3-1} with 
\begin{equation}\label{eq:3-2}
  \Theta_{r,k+1} =\big\{ 1 \leq  i \leq d_r {T_r}^{k}:\, I_{r,i}^{k+1} \subset B_{r,k+1}   \big\},
\end{equation}
we have
\begin{equation}\label{eq:wdt-relation}
	\Theta_{r,k+1}\subset \wdt{\Theta}_{r,k}.
\end{equation}


Now we prove the following simple lemma.
\begin{lem}\label{lem:equivalent-cr}
	Fix $1 \leq r \leq m$. The following statements are equivalent.
	\begin{enumerate}
		\item $ \bigcup_{n \in \Lambda_r}D_n=\bigcup_{n \in \Lambda_r} I_n$.
		\item $B_{r,1}=B_{r,2}$.
		\item $B_{r,k}=B_{r,k+1}$  for every $k \in \mathbb{Z}^+$.
		\item $K_r=B_{r,1}$.
		\item $ \Theta_{r,k+1}=\wdt{\Theta}_{r,k}$ for every $k \in \mathbb{Z}^+$.
	\end{enumerate}
\end{lem}
\begin{proof}
	By definition, (1) is equivalent to $B_{r,0}=B_{r,1}$. Notice that $\{B_{r,k}\}_{k\geq 0}$ is decreasing with respect to $k$. Hence, from \eqref{eq:B-rk-RecRelation} and $K_r=\bigcap_{k=1}^\infty B_{r,k}$, it is clear that
	\[(1) \Rightarrow (2) \Leftrightarrow (3)\Leftrightarrow (4).\]
	Meanwhile, it follows from \eqref{eq:3-1} and \eqref{eq:3-2} that $(3)\Leftrightarrow (5)$.


	From above arguments, it suffices to prove that $(2) \Rightarrow (1)$. Assume that (2) holds. Then for any $n \in \Lambda_r$, 
	we have $I_n \cap B_{r,1}=I_n= L_n(D_n)$ and 
	\[
	I_n = I_n \cap B_{r,2}
	= L_n ( D_n\cap  B_{r,1}  ) 
	= L_n\Big( \bigcup_{j\in \Lambda_r: I_j\subset D_n} I_j  \Big).
	\]
	Notice that $L_n$ is a bijection.
	Thus, for each $n\in \Lambda_r$, 
	\[D_n=\bigcup_{j\in \Lambda_r: I_j\subset D_n} I_j\subset \bigcup_{j\in \Lambda_r} I_j,\] 
	and therefore $\bigcup_{n\in \Lambda_r} D_n\subset \bigcup_{j\in \Lambda_r} I_j $. On the other hand, $\bigcup_{j\in \Lambda_r} I_j\subset \bigcup_{n\in \Lambda_r} D_n$ always holds since $\Lambda_r$ is a strongly connected component of $\{1,2,\ldots,N\}$. Thus (1) holds.	
\end{proof}

\subsection{The irreducibility of restricted vertical scaling matrices}
We recall some notations and definitions in matrix analysis \cite{HJ13}.
Given a matrix $A=(a_{ij})_{n\times n}$, we say $A$ is \emph{nonnegative}, denoted by $A\geq 0$ , if $a_{ij}\geq0$ for all $i$ and $j$.
 Similarly, given $\bfu=(u_1,\ldots,u_n),\bfv=(v_1,\ldots,v_n)\in \bR^n$, we write $\bfu \geq \bfv$ if $u_i\geq v_i$ for all $i$. We say $\bfu$ is positive, denoted by $\bfu>0$, if $u_{i}>0$ for all $i$.

 A nonnegative matrix $A=(a_{ij})_{n\times n}$ is called \emph{irreducible} if for any $i,j\in \{1,\ldots,n\}$, there exists a finite sequences $i_0,\ldots,i_t\in\{1,\ldots,n\}$ such that $i_0=i,i_t=j$, and $a_{i_{\ell-1},i_{\ell}}>0$ for all $1\leq \ell\leq t-1$. Please see \cite{HJ13} for details.

\begin{lem}\label{lem3.3}
 Let $1\leq r\leq m$ and $k\in \Z^+$. Then for any $i,j\in \Theta_{r,k}$, there exists a finite sequence $\{i_\ell\}_{\ell=0}^t$ in $\Theta_{r,k}$ such that $i_0=j$, $i_t=i$ and $I^k_{r,i_{\ell-1}}\subset D_{r,i_\ell}^k$ for all $1\leq \ell \leq t$.
\end{lem}

 \begin{proof}
 	We prove the lemma by induction on $k$. Fix $1\leq r\leq m$. 
 	
 	First we assume that $k=1$. Notice that $\Lambda_r=\{a_{r,1}, \ldots,a_{r,d_r}\}$ 
 	and $\Theta_{r,1}=\{1,\ldots,d_r\}$. From \eqref{eq:Irt-One} and Lemma~\ref{lem:Theta-simple-facts}(5), for all $1\leq i\leq d_r$, we have $I_{r,i}^1=I_{a_{r,i}}$ and $D_{r,i}^1=D_{a_{r,i}}$. Combining this with the fact that $\Lambda_r=\{a_{r,1}, \ldots, a_{r,d_r}\}$ is strong connected in $\{1,\ldots,N\}$, we know the statement holds for $k=1$.
 	
 	Assume that the statement holds for $k=p$, where $p$ is a positive integer. Arbitrarily pick $i',j'\in \Theta_{r,p+1}$. Let $j$ be the element in $\{1,\ldots,d_r {T_r}^{p-1}\}$ satisfying $I_{r,j'}^{p+1} \subset I_{r,j}^p$. Then $I_{r,j}^p\cap B_{r,p+1}\supset I_{r,j'}^{p+1}$. Thus from Lemma~\ref{lem:Theta-simple-facts}(2), $j \in \Theta_{r,p}$.
 	From $i'\in \Theta_{r,p+1}$, there exists $n_0\in \Lambda_r$ and $E\in \mB_{r,p}$ with $E\subset D_{n_0}$, such that $I_{r,i'}^{p+1}=L_{n_0}(E)$. From $E\in \mB_{r,p}$, there exists $1\leq i\leq d_r {T_r}^{p-1}$ such that $E=I_{r,i}^p$. Thus, $I_{r,i'}^{p+1}=L_{n_0}(I_{r,i}^p)$ so that
 	$I_{r,i}^p=D_{r,i'}^{p+1}$. Noting that from $I_{r,i}^p=E\in \mB_{r,p}$, we have $i\in \Theta_{r,p}$. 	
 	By the inductive assumption, there exists a finite sequence $\{\ell_t\}_{t=0}^{q}$ such that
 	$\ell_0=j$, $\ell_{q}=i$ and $I_{r,\ell_{t-1}}^p \subset D_{r,\ell_t}^p$ for all $1\leq t\leq q$.
 	
 	For each $1\leq t\leq q$, let $n_t$ be the unique element in $\Lambda_r$ such that $I_{r,\ell_t}^p\subset I_{n_t}$. It follows from $I_{r,\ell_{t-1}}^p\subset D_{r,\ell_t}^p$ that there exists $\ell_t'\in\{(\ell_t-1)T_r+1,\ldots,\ell_t T_r\}$ such that $L_{n_t}(I_{r,\ell_{t-1}}^p)=I_{r,\ell_t'}^{p+1}$. 
	 Thus $I_{r,\ell_{t-1}}^p=D_{r,\ell_t'}^{p+1}$ and $I_{r,\ell_t'}^{p+1}\subset I_{r,\ell_t}^p$ for all $1\leq t\leq q$. In other words, $I_{r,\ell_{t-1}'}^{p+1}\subset I_{r,\ell_{t-1}}^p$ for all $2\leq t\leq q+1$. 
 	As a result, for all $2\leq t\leq q$,
 	\begin{equation}\label{eq:lem-3-1-1}
 		I_{r,\ell_{t-1}'}^{p+1} \subset I_{r,\ell_{t-1}}^p=D_{r,\ell_t'}^{p+1}.
 	\end{equation}	
 	Define $\ell_0'=j'$ and $\ell_{q+1}'=i'$. From $\ell_0=j$, $\ell_q=i$ and definitions of $i$, $j$, \eqref{eq:lem-3-1-1} still holds for $t=1$ and $t=q+1$. Thus the statement holds for $k=p+1$. 	
 \end{proof}

 \begin{proof}[Proof of Theorem~\ref{thm:irr}]
 	Fix $1\leq r\leq m$. First we prove (1). Fix $k\geq 1$.	
 	From Lemma~\ref{lem3.3}, for any $i,j\in \Theta_{r,k}$, there exists a finite sequence $\{i_\ell\}_{\ell=0}^t$ in $\Theta_{r,k}$ such that $i_0=j$, $i_t=i$ and 
 	$I^k_{r,i_{\ell-1}}\subset D_{r,i_\ell}^k$ for all $1\leq \ell \leq t$. Hence, by using the assumption that $S_n$ is positive on $D_n$ for all $n\in \Lambda_r$, we know from definitions of vertical scaling matrices that
 	\begin{equation}\label{eq:thm-irr-1}
 	  \ups_{r,i_\ell,i_{\ell-1}}^{k} \geq \lows_{r,i_\ell,i_{\ell-1}}^{k}>0, \quad 1\leq q\leq t.
 	\end{equation}
    As a result, both $\lowM_{r,k}|_{\Theta_{r,k}}$ and $\upM_{r,k}|_{\Theta_{r,k}}$ are irreducible. 

 	Now we prove (2). Assume that $S_n$ is positive on $D_n\cap K_r$ for all $n\in\Lambda_r$. Then from the fact that $S_n$ is continuous on $D_n$ and $K_r=\bigcap_{k=0}^\infty B_{r,k}$ is the intersection of decreasing compact sets, there exists $k_1\in \bZp$ such that $S_n$ is positive on $D_n\cap B_{r,k}$ for all $n\in\Lambda_r$ and $k\geq k_1$. 
	 Fix $k\geq k_1+1$.	
 	From Lemma~\ref{lem3.3}, for any $i,j\in \Theta_{r,k}$, there exists a finite sequence $\{i_\ell\}_{\ell=0}^t$ in $\Theta_{r,k}$ such that $i_0=j$, $i_t=i$ and 
 	$I^k_{r,i_{\ell-1}}\subset D_{r,i_\ell}^k$ for all $1\leq \ell \leq t$. From Lemma~\ref{lem:Theta-simple-facts}(5), $D_{r,i_\ell}^k\subset B_{r,k-1}$ for all $1\leq \ell \leq t$. Hence, \eqref{eq:thm-irr-1} still holds since $k\geq k_1+1$. Thus both $\lowM_{r,k}|_{\Theta_{r,k}}$ and $\upM_{r,k}|_{\Theta_{r,k}}$ are irreducible.
	 By letting $k_0=k_1+1$, (2) holds. 
\end{proof}

\subsection{Convergence of spectral radii of restricted vertical scaling matrices}

The following four lemmas are well known. Please see \cite[Chapter 8]{HJ13} for details.

\begin{lem}\label{lem:psm}
Let $A$ be a nonnegative matrix. 
If $A'$ is a principal submatrix of $A$, then $\rho(A') \leq \rho(A)$.

\end{lem}

\begin{lem}\label{lem:SptlRadIncr-Nonngtv}
	Let $A$ and $A'$ be two nonnegative matrices with $A\leq A'$. Then $\rho(A)\leq \rho(A')$.
\end{lem}

\begin{lem}[Perron-Frobenius Theorem]\label{th:PF}
Let $A=(a_{ij})_{n\times n}$ be an irreducible nonnegative matrix. Then $\rho(A)$ is a positive eigenvalue of $A$ and has a positive eigenvector.
\end{lem}

\begin{lem}\label{lem:PFN}
Let $A$ be a nonnegative matrix. 
Then $\rho(A)$ is an eigenvalue of $A$ and there
is a nonnegative nonzero vector $\bfx$ such that $A\bfx  = \rho(A)\bfx$.

\end{lem}


For each $1\leq r\leq m$ and $k \geq 1$, we introduce a $d_r {T_r}^k \times d_r {T_r}^k$ matrix $\upM_{r,k}^*$ as follows. Given $1\leq i ,j \leq d_r {T_r}^{k}$, let $n$ be the unique element in $\Lambda_r$ such that $I_{r,i}^{k+1} \subset I_n$. We define
\begin{equation*}
	\big(\upM_{r,k}^*\big)_{i,j}= 
	\begin{cases}
		\max \big\{ \big| S_n(x) \big|: x \in D_{r,i}^{k+1}\big\} ,\quad & \mbox{if } I_{r,j}^{k+1}\subset D_{r,i}^{k+1}, \\
		0, \quad & \mbox{otherwise}.
	\end{cases}
\end{equation*}

	\begin{lem}\label{lem:new-eq}
	For any $1 \leq r \leq m$ and $k \geq 1$, we have
		$$
		\rho(\upM_{r,k}^*|_{\wdt{\Theta}_{r,k}}) = \rho(\upM_{r,k}^*|_{\Theta_{r,k+1}}).	
		$$
	\end{lem}

\begin{proof}
	
	Notice that both $\upM_{r,k}^*|_{\Theta_{r,k+1}}$ and $\upM_{r,k}^*|_{\wdt {\Theta}_{r,k}}$ are principal submatrices of $\upM_{r,k}^*$. From \eqref{eq:wdt-relation}, $\wdt{\Theta}_{r,k} \supset {\Theta}_{r,k+1}$. Thus
	$\upM_{r,k}^*|_{\Theta_{r,k+1}}$ is a principal submatrix of $\upM_{r,k}^*|_{\wdt {\Theta}_{r,k}}$.
	As a result, by Lemma ~\ref{lem:psm},
	$$
	\rho(\upM_{r,k}^*|_{\wdt{\Theta}_{r,k}}) \geq \rho(\upM_{r,k}^*|_{\Theta_{r,k+1}}).	
	$$
	
	Now we prove the another inequality. 
	Without loss of generality, we may assume that $\lambda=\rho(\upM_{r,k}^*|_{\wdt{\Theta}_{r,k}} )>0$ and ${\Theta}_{r,k+1} \subsetneq \wdt{\Theta}_{r,k}$.
	From Lemma~\ref{lem:PFN}, 
	$\lambda$ is an eigenvalue of $\upM_{r,k}^* |_{\wdt{\Theta}_{r,k}}$ and there is a nonnegative nonzero vector 
	$ \wdt\bfu=(\wdt{u}_{i})_{i\in {\wdt{\Theta}_{r,k}}}$ such that
	\begin{equation}\label{eq:3-3}
	  (\upM_{r,k}^* |_{\wdt{\Theta}_{r,k}}) \wdt\bfu^T =\lambda \wdt\bfu^T.		
	\end{equation}

	Fix $i \in \wdt{\Theta}_{r,k} \backslash {\Theta}_{r,k+1}$. Since $i \in \wdt{\Theta}_{r,k}$, there exists $i' \in \Theta_{r,k}$ such that $I_{r,i}^{k+1}\subset I_{r,i'}^k$. It follows from $i\not\in {\Theta}_{r,k+1}$ that $D_{r,i}^{k+1}\not\in \mB_{r,k}$. Thus by Lemma~\ref{lem:Theta-simple-facts}(5), $(D_{r,i}^{k+1}\cap B_{r,k})^\circ=\emptyset$.
	Hence, for any $j \in \wdt{\Theta}_{r,k}$, by letting $j' \in \Theta_{r,k}$ the unique element satisfying $I_{r,j}^{k+1}\subset I_{r,j'}^k$, we have 
	$(D_{r,i}^{k+1}\cap I_{r,j'}^k )^\circ=\emptyset$. As a result, 
	$
	I_{r,j}^{k+1}  \not \subset D_{r,i}^{k+1}.
	$
	Hence
	\begin{equation*}
		(\upM_{r,k}^*)_{i,j}=0 ,\quad  j \in \wdt{\Theta}_{r,k}.
	\end{equation*}
	Thus from \eqref{eq:3-3} and $\lambda>0$, 
	we have 
	\begin{equation}\label{eq:wdtUiEqualZero}
	    \wdt{u}_{i}=0, \quad  i \in \wdt{\Theta}_{r,k} \backslash {\Theta}_{r,k+1}. 
    \end{equation}
    As a result, the vector $\bfu=(u_i)_{i \in \Theta_{r,k+1}}$ defined by $\bfu=\wdt{\bfu}|_{\Theta_{r,k+1}}$ is also a nonnegative nonzero vector.
	Furthermore, for any $i \in \Theta_{r,k+1}$, by \eqref{eq:wdtUiEqualZero} and the definitions of $\bfu$ and $\wdt\bfu$,
	\[
	   \sum_{j \in {\Theta}_{r,k+1}} (\upM_{r,k}^* )_{i,j} {u}_{j}
	   =\sum_{j \in \wdt{\Theta}_{r,k}} (\upM_{r,k}^* )_{i,j}\wdt{u}_{j}
	   =\lambda \wdt{u}_{i}=\lambda u_{i}. 
	\]
	From the arbitrariness of $i$, we have $\lambda \bfu^T=(\upM_{r,k}^* |_{\Theta_{r,k+1}}) \bfu^T$.
	Thus, $\lambda$ is an eigenvalue of $\upM_{r,k}^*|_{\Theta_{r,k+1}} $, which implies $\rho(\upM_{r,k}^* |_{\wdt{\Theta}_{r,k}})=\lambda \leq \rho(\upM_{r,k}^*|_{\Theta_{r,k+1}})$.
\end{proof}



Given $1\leq r\leq m$, $k \in \mathbb{Z}^+$ and $\Omega \subset \{ 1 ,\ldots, d_r {T_r}^{k-1}\}$, we define 
\[
  \wdt{\Omega}(r)= 
  \big\{ (i-1)T_r+\ell : \, i \in \Omega , 1 \leq \ell \leq T_r    \}.
\]
We write $\wdt{\Omega}=\wdt{\Omega}(r)$ if there is no confusion. It is clear that $\wdt{\Omega}(r)=\wdt{\Theta}_{r,k}$ if $\Omega=\Theta_{r,k}$.

\begin{lem}\label{lem:3-rho-key-lem-1}
For any $1\leq r\leq m$, $k \geq 1$ and $\Omega \subset \{ 1, \ldots, d_r {T_r}^{k-1} \}$, we
have
$$ \rho(\upM_{r,k}|_{\Omega}) = \rho( \upM_{r,k}^*|_{\wdt\Omega}).  $$
\end{lem}

\begin{proof}

Firstly, we are going to prove $\rho(\upM_{r,k} |_\Omega)\geq \rho(\upM_{r,k}^* |_{\wdt{\Omega}})$.
Write $\lambda=\rho(\upM_{r,k}^*|_{\wdt{\Omega}} )$.
From Lemma~\ref{lem:PFN}, 
$\lambda$ is an eigenvalue of $\upM_{r,k}^* |_{\wdt{\Omega}}$ and there is a nonnegative nonzero vector 
$ \wdt\bfu=(\wdt{u}_{i'})_{i'\in \wdt{\Omega}}$ such that $(\upM_{r,k}^* |_{\wdt{\Omega}}) \wdt\bfu^T =\lambda \wdt\bfu^T$.

Define a vector $\bfu=(u_i)_{i\in\Omega}$ by letting
\[
{u}_i=\sum_{i'=(i-1)T_r+1}^{iT_r} \wdt{u}_{i'}, \quad i\in\Omega.
\]
Obviously, $\bfu$ is also nonnegative and nonzero.

Now we fix $i\in \Omega$. For any $(i-1)T_r<i'\leq iT_r$, by the definition of $\wdt\bfu^T$,
\[
  \lambda \wdt{u}_{i'} =\sum_{j'\in \wdt{\Omega}} \big(\upM_{r,k}^*\big)_{i',j'}  \wdt{u}_{j'}.
\]
Let $n$ be the unique element in $\{1,2,\ldots,N\}$ satisfying $I_{r,i'}^{k+1}\subset I_{r,i}^k\subset I_n$.

In the case that $D_{r,i'}^{k+1}\not\subset \bigcup_{p\in \Lambda_r} I_p$, it is clear that for all $j'\in\wdt{\Omega}$, we have $I^{k+1}_{r,j'}\not\subset D_{r,i'}^{k+1}$ so that $\big(\upM_{r,k}^*\big)_{i',j'}=0$. Hence, $\lambda \wdt{u}_{i'}=0$.  

In the case that $D_{r,i'}^{k+1}\subset \bigcup_{p\in \Lambda_r} I_p$, from Lemma~\ref{lem:Theta-simple-facts}(5), there exists $\ell\in\{1,\ldots,d_r{T_r}^{k-1}\}$ such that $I_{r,\ell}^k=D_{r,i'}^{k+1}$. If $\ell\not\in\Omega$, then for all $j'\in\wdt{\Omega}$, we have $I^{k+1}_{r,j'}\not\subset I_{r,\ell}^k= D_{r,i'}^{k+1}$ so that $\big(\upM_{r,k}^*\big)_{i',j'}=0$, which gives $\lambda \wdt{u}_{i'}=0$. If $\ell\in\Omega$, then 
$\big(\upM_{r,k}^*\big)_{i',j'}\not=0$ if and only if $I^{k+1}_{r,j'}\subset D_{r,i'}^{k+1}=I^k_{r,\ell}$, which is equivalent to $(\ell-1)T_r<j'\leq \ell T_r$. Furthermore, if $(\ell-1)T_r<j'\leq \ell T_r$, then
\[
\big(\upM_{r,k}^*\big)_{i',j'}=\max \big\{|S_n(x)|:\, x\in D^{k+1}_{r,i'} \big\}
=\max\big\{|S_n(x)|:\, x\in I^k_{r,\ell}\big\}=\ups_{r,i,\ell}^k,
\]
where the last equality follows from the fact that $I_{r,\ell}^k=D_{r,i'}^{k+1}\subset D_{r,i}^k$. Hence
\[
\lambda \wdt{u}_{i'} = \ups_{r,i,\ell}^k \sum_{j'=(\ell-1)T_r+1}^{ \ell T_r } \wdt{u}_{j'}
=\ups_{r,i,\ell}^k  {u}_{\ell}.
\]


Notice that the map
$i'\mapsto \ell$ defined by $I^{k}_{r,\ell}=D^{k+1}_{r,i'}$ is a bijection from 
\[
\big\{i':\, (i-1)T_r<i'\leq iT_r \mbox{ and } D_{r,i'}^{k+1}=I_{r,\ell}^k \mbox{ for some } \ell\in \Omega\big\}
\] 
to $\{\ell\in \Omega: I_{r,\ell}^k \subset D_{r,i}^k \}$. Thus from the above arguments, 
\[
  \lambda {u}_i=\sum_{i'=(i-1)T_r+1}^{iT_r} \lambda \wdt{u}_{i'}=\sum_{\ell\in \Omega: I_{r,\ell}^k \subset D_{r,i}^k} \ups_{r,i,\ell}^k u_{\ell}=
   \sum_{\ell\in \Omega} (\upM_{r,k})_{i,\ell}  u_{\ell},
\]
where the last equality follows from the definitions of $\upM_{r,k}$ and $\ups_{r,i,\ell}^k$.

From the arbitrariness of $i$, we have $\lambda \bfu^T=(\upM_{r,k} |_\Omega) \bfu^T$.
Thus, $\lambda$ is an eigenvalue of $\upM_{r,k}|_\Omega $, which implies $\rho(\upM_{r,k}^* |_{\wdt{\Omega}} )=\lambda \leq \rho(\upM_{r,k}|_\Omega)$.

In the second part, we are going to prove $\rho(\upM_{r,k}|_\Omega)\leq \rho(\upM_{r,k}^* |_{\wdt{\Omega}})$. Without loss of generality, we assume that 
$\mu=\rho(\upM_{r,k}|_\Omega)>0$. 
From Lemma~\ref{lem:PFN}, there is a nonnegative nonzero vector $\bfv=(v_i)_{i\in\Omega}$ such that $\big( \upM_{r,k} |_\Omega \big)\bfv^T =\mu \bfv^T$.

For any $i\in \Omega$ and $(i-1)T_r<i'\leq iT_r$, if there exists $\ell \in \Omega$ such that $I_{r,\ell}^k=D_{r,i'}^{k+1}$, then we define $\wdt{v}_{i'}=	\ups_{r,i,\ell}^k v_{\ell}$; otherwise, we define $\wdt{v}_{i'}=0$. We remark that in the former case, $I_{r,\ell}^k\subset D_{r,i}^k$. Hence,
it is straightforward to see that
\begin{equation}\label{eq:3-eigen-vv-sum}
	\sum_{i'=(i-1)T_r+1}^{iT_r} \wdt{v}_{i'}=\sum_{\ell\in \Omega: I_{r,\ell}^k \subset D_{r,i}^k} \ups_{r,i,\ell}^k v_{\ell}=\sum_{\ell\in \Omega} (\upM_{r,k})_{i,\ell}  v_{\ell}=\mu v_i.
\end{equation}
As a result, the vector $\wdt{\bfv}=(\wdt{v}_{i'})_{i'\in\wdt{\Omega}}$ is nonnegative and nonzero.

Fix $i\in\Omega$ and $i'\in\{(i-1)T_r+1,\ldots,iT_r\}$. Assume that $D_{r,i'}^{k+1}\not\subset \bigcup_{p\in \Lambda_r} I_p$ or $D_{r,i'}^{k+1}=I_{r,\ell}^k$ for some $\ell\not\in \Omega$. By definition, $\wdt{v}_{i'}=0$.
Notice that in this case, for all $j'\in \wdt{\Omega}$, we have $I^{k+1}_{r,j'}\not\subset D_{r,i'}^{k+1}$ so that $\big(\upM_{r,k}^*\big)_{i',j'}=0$. Hence, 
\[
\sum_{j'\in \wdt{\Omega}} \big(\upM_{r,k}^*\big)_{i',j'} \wdt{v}_{j'}=0=\mu \wdt{v}_{i'}.
\]

Assume that $D_{r,i'}^{k+1}=I_{r,\ell}^k$ for some $\ell\in \Omega$. By arguments in the first part of the proof, $\big(\upM_{r,k}^*\big)_{i',j'}\not=0$ if and only if $(\ell-1)T_r<j'\leq \ell T_r$, and if this holds, $\big(\upM_{r,k}^*\big)_{i',j'}=\ups_{r,i,\ell}^k$. Thus,
\[
\sum_{j'\in \wdt{\Omega}} \big(\upM_{r,k}^*\big)_{i',j'} \wdt{v}_{j'}=\sum_{j'=(\ell-1)T_r+1}^{\ell T_r} \big(\upM_{r,k}^*\big)_{i',j'} \wdt{v}_{j'}=\ups_{r,i,\ell}^k \sum_{j'=(\ell-1)T_r+1}^{ \ell T_r }  \wdt{v}_{j'}. 
\]
Combining this with \eqref{eq:3-eigen-vv-sum}, 
\[
\sum_{j'\in \wdt{\Omega}} \big(\upM_{r,k}^*\big)_{i',j'} \wdt{v}_{j'}=\ups_{r,i,\ell}^k \mu v_{\ell}=\mu \wdt{v}_{i'}.
\]

From the arbitrariness of $i'$, we have $\upM_{r,k}^*  \wdt{\bfv}^T =\mu \wdt{\bfv}^T$ so that $\mu$ is an eigenvalue of $\upM_{r,k}^*$. Hence, 
$\rho(\upM_{r,k})=\mu\leq \rho(\upM_{r,k}^*)$.
\end{proof}

\begin{proof}[Proof of Theorem~\ref{thm:rho-up-sub}]
Fix $1 \leq r \leq m$. First we prove (1). Fix $k \in \mathbb{Z}^+$. From the definitions of $\upM_{r,k}^*$ and $\upM_{r,k+1}$, it is easy to see that $\upM_{r,k}^* \geq \upM_{r,k+1}\geq 0$. Thus from Lemma~\ref{lem:SptlRadIncr-Nonngtv},
\[
	\rho( \upM_{r,k}^*|_{\Theta_{r,k+1}}) \geq \rho(\upM_{r,k+1}|_{\Theta_{r,k+1}}).
\]
Combining this with Lemmas~\ref{lem:new-eq} and \ref{lem:3-rho-key-lem-1}, 
\[ 
  \rho(\upM_{r,k}|_{\Theta_{r,k}})
  = \rho( \upM_{r,k}^*|_{\wdt {\Theta}_{r,k}})
  = \rho(\upM_{r,k}^*|_{{\Theta_{r,k+1}}}) 
   \geq \rho(\upM_{r,k+1}|_{{\Theta_{r,k+1}}}).
\]
Thus $\{\rho(\upM_{r,k}|_{\Theta_{r,k}})\}_{k=1}^\infty$ is decreasing. As a result, there exists a nonnegative constant $\uprho_r$ such that
$$ \uprho_r=\lim_{k  \to \infty}\rho(\upM_{r,k}|_{\Theta_{r,k}}).$$

Now we prove (2). For each $k \in \Z^+$, we define
	\begin{equation*}
		\big(\lowM_{r,k}^*\big)_{i,j}= 
		\begin{cases}
			\min \big\{ \big| S_n(x) \big|: x \in D_{r,i}^{k+1}\big\} ,\quad & \mbox{if } I_{r,j}^{k+1}\subset D_{r,i}^{k+1}, \\
			0, \quad & \mbox{otherwise}.
		\end{cases}
	\end{equation*}
	Then $\lowM_{r,k}^* \leq \lowM_{r,k+1}$ so that
	$
	\rho( \lowM_{r,k}^*|_{\Theta_{r,k+1}}) \leq \rho(\lowM_{r,k+1}|_{\Theta_{r,k+1}}).
	$
	Using similar arguments in the proofs of Lemmas ~\ref{lem:new-eq} and \ref{lem:3-rho-key-lem-1},
	we can obtain
	\[ 
	  \rho(\lowM_{r,k}^*|_{\wdt {\Theta}_{r,k}}) = \rho(\lowM_{k}^*|_{{\Theta_{r,k+1}}}) \quad \mbox{and} \quad
	  \rho(\lowM_{r,k}|_{\Theta_{r,k}}) = \rho( \lowM_{r,k}^*|_{\wdt {\Theta}_{r,k}}).
	\]
	Hence
	$$	 \rho(\lowM_{r,k}|_{\Theta_{r,k}}) =\rho( \lowM_{r,k}^*|_{\wdt {\Theta}_{r,k}})=\rho( \lowM_{r,k}^*|_{{\Theta_{r,k+1}}}) \leq \rho(\lowM_{r,k+1}|_{{\Theta_{r,k+1}}}).$$
	As a result, the sequence $\{ \rho(\lowM_{r,k}) \}_{k=1}^\infty$ is increasing.

In the sequel of the proof, we assume that $S_n$ is positive on $D_n \cap K_r$ for all $n\in \Lambda_r$. Using similar arguments in the proof of Theorem~\ref{thm:irr}(2), we know that there exists $k_1\in \bZp$ such that $S_n$ is positive on $D_n \cap B_{r,k_1}$ for all $n\in \Lambda_r$. Define
\[
  \lows_r=\min\{ |S_n(x)|:\, x\in  D_n\cap B_{r,k_1},  n \in \Lambda_r \}.
\]
Then $\lows_r>0$.
Notice that $S_n$ is uniformly continuous on $D_n$ for all $n\in \Lambda_r$.
Fix $\vep>0$, there exists a positive integer $k_2\geq k_1$ such that 
\begin{align*}
	\sup\big\{ |S_n(x')-S_n(x'')|:\, x',x''\in D_{r,i}^k\big\} \leq \lows_r\vep
\end{align*}
for all $k>k_2$ and all $1\leq i\leq d_r {T_r}^{k-1}$.



Assume that $k>k_2$ and $i,j \in \Theta_{r,k}$. 
In the case that $I_{r,j}^k \subset D_{r,i}^k $, 
let $n$ be the unique element in $\Lambda_r$ such that $I_{r,i}^k \subset I_n$.
Then
\begin{align*}
	0\leq  \ups_{r,i,j}^k-\lows_{r,i,j}^k 
	=	&\sup\big\{ |S_n(x')-S_n(x'')|:\, x',x''\in I_{r,j}^k\big\} \\
	\leq 	&\sup\big\{ |S_n(x')-S_n(x'')|:\, x',x''\in D_{r,i}^k\big\} 
    \leq \lows_r\vep \leq  \lows_{r,i,j}^k  \vep.
\end{align*}
In the case that $I_{r,j}^k \not\subset D_{r,i}^k $, we have $\ups_{r,i,j}^k=\lows_{r,i,j}^k=0$.
Hence, for all $k>k_2$ and $i,j \in \Theta_{r,k}$,
\begin{align*}
	\lows_{r,i,j}^k \leq \ups_{r,i,j}^k \leq (1 +\vep)\lows_{r,i,j}^k .
\end{align*}
From the definitions of $\upM_{r,k}|_{\Theta_{r,k}}$ and $\lowM_{r,k}|_{\Theta_{r,k}}$, 
$$\lowM_{r,k}|_{\Theta_{r,k}} \leq \upM_{r,k}|_{\Theta_{r,k}} \leq (1 +\vep) \lowM_{r,k}|_{\Theta_{r,k}}.$$
Thus
$\rho (\lowM_{r,k}|_{\Theta_{r,k}}) \leq \rho( \upM_{r,k}|_{\Theta_{r,k}}) \leq (1 +\vep) \rho(\lowM_{r,k}|_{\Theta_{r,k}})$.
By letting $k$ tend to infinity, 
\[
 \frac{\uprho_r}{1+\vep} \leq \varliminf_{k  \to \infty}\rho(\lowM_{r,k}|_{\Theta_{r,k}})
 \leq \varlimsup_{k  \to \infty}\rho(\lowM_{r,k}|_{\Theta_{r,k}}) \leq  \uprho_r.
\]
From the arbitrariness of $\vep$, $\lim_{k  \to \infty}\rho(\lowM_{r,k}|_{\Theta_{r,k}})=\uprho_r$.
\end{proof}

\section{Some lemmas for box dimension estimate}

\subsection{The relationship between box dimension and oscillation sum}
For any $k_1, k_2\in\mathbb{Z}$ and $\varepsilon>0$, we call $[k_1\varepsilon,(k_1+1)\varepsilon] \times [k_2\varepsilon,(k_2+1)\varepsilon]$ an $\varepsilon$-coordinate square in $\mathbb{R}^2$. Let $E$ be a bounded set in $\mathbb{R}^2$ and $\mathcal{N}_{\vep}(E) $ the least number of $\varepsilon$-coordinate squares that cover $E$.
We define
\begin{equation}\label{eq:box-dim-def}
	   \updimB E=\varlimsup_{\varepsilon\to 0+}\frac{\log \mathcal{N}_{\varepsilon}(E)}{\log1/\varepsilon}
	\quad\text{and}	\quad
	\lowdimB E=\varliminf_{\varepsilon\to 0+}\frac{\log \mathcal{N}_{\varepsilon}(E)}{\log1/\varepsilon},
\end{equation}
and call them the \emph{upper box dimension} and the \emph{lower box dimension} of $E$, respectively. If $\updimB E=\lowdimB E$, then we use $\dim_B E$ to denote the common value and call it the \emph{box dimension} of $E$.
It is well known that in \eqref{eq:box-dim-def}, it is enough to consider limits as $\varepsilon$ tends to $0$ through any decreasing sequence $\{\varepsilon_k\}$ such that $\varepsilon_{k+1}\geq c\varepsilon_k$ for some constant $0<c<1$, in particular for $\varepsilon_k=\alpha c^k$, where $\alpha>0$ is a constant.
It is well known that $\lowdimB  E\geq 1$ when $E$ is the graph of a continuous function on a closed interval of $\mathbb{R}$. Please see ~\cite{Fal14} for details.

Let $g$ be a continuous function on $J=[a,b]$. For each $p\in \mathbb{Z}^+$ and $1\leq r\leq m$,
we define
\[
O_{r,p}(g, J)
=\sum\limits_{\ell=1}^{{T_r}^p} O\Big(g,\Big[a+\frac{\ell-1}{{T_r}^p}(b-a) , a+\frac{\ell}{{T_r}^p}(b-a) \Big]\Big),
\]
where we use $O(g,U)$ to denote the \emph{oscillation} of $g$ on $U\subset J$, that is,
\begin{equation*}
	O(g,U)= \sup\limits_{x',x'' \in U}|g({x}')-g({x}'')|.
\end{equation*}
We call $O_{r,p}(g, J)$ the \emph{oscillation sum} of $g$ on $J$ with level $p$ and associated with $r$.
It is clear that $\{O_{r,p}(g,J)\}_{p=1}^\infty$ is increasing with respect to $p$. Thus $\lim_{p\to \infty} O_{r,p}(g,J)$ always exists.
Recall that $\Var(g,J)$ is the classical total variation of $g$ on $J$. We have the following fact.
\begin{lem}[\cite{JR24}, Lemma~4.2]\label{lem:VarOinfty}
	Let $g$ be a continuous function on a closed interval $J$. Then for any $1\leq r\leq m$, $\lim_{p\to \infty} O_{r,p}(g,J)=\Var(g,J)$.
\end{lem}

The following lemma presents a useful method to estimate the upper and lower box dimensions of the graph of a continuous function by its oscillation sum. Similar results can be found in \cite{Fal14,JR23,JR24,RSY09}.

\begin{lem}\label{lem:box-dim-new-x1}
	Assume that $g$ is a continuous function on a closed interval $J$. Then for any $1 \leq r \leq m$, 
	\begin{align}
		&\lowdimB \Gamma g = 1+\varliminf_{p\to\infty}\frac{\log \big( O_{r,p}(g,J)+1\big)}{p\log T_r}, 	\quad \mbox{and} \label{eq:box-dim-cal-tmp-3} \\
		&\updimB \Gamma g = 1+\varlimsup_{p\to\infty} \frac{\log \big(O_{r,p}(g,J)+1\big)}{p\log T_r}.
		\label{eq:box-dim-cal-tmp-4}
	\end{align}
\end{lem}
\begin{proof}
	Assume that $J=[a,b]$. For any $k_1, k_2\in\mathbb{Z}$ and $\varepsilon>0$, we call $[a+k_1\varepsilon,a+(k_1+1)\varepsilon] \times [k_2\varepsilon,(k_2+1)\varepsilon]$ a generalized $\varepsilon$-coordinate square with center $(a,0)$ in $\mathbb{R}^2$. We use $\mathcal{N}_{\varepsilon}^{(a,0)}(E)$ to denote the least number of generalized $\varepsilon$-coordinate squares with center $(a,0)$ to cover $E$. It is clear that in \eqref{eq:box-dim-def}, we can replace $\mathcal{N}_{\varepsilon}(E)$ by $\mathcal{N}_{\varepsilon}^{(a,0)}(E)$.
	
	Let $\vep_p=|J|/{T_r}^{p}$ for all $p \in \Z^+$.
	It is straightforward to check that
	\[
	    \max\{{T_r}^p, \vep_p^{-1} O_{r,p}(g,J)\} \leq \mathcal{N}_{\varepsilon_p}^{(a,0)}(\Gamma g)  \leq \vep_p^{-1}  O_{r,p}(g,J) +2{T_r}^p.
	\]
	Thus, by using $\max\{x,y\}\geq (x+y)/2$ for all $x,y\in \bR$, we have
		\begin{equation}\label{eq:boxxx1}
			\frac{1}{2}\vep_p^{-1} \big(O_{r,p}(g,J) +|J|\big) \leq \mathcal{N}_{\varepsilon_p}^{(a,0)}(\Gamma g)  \leq \vep_p^{-1} \big(O_{r,p}(g,J) +2|J|\big).
		\end{equation} 
	As a result,
	\begin{align*}
		\updimB E 
		&= \varlimsup_{p\to \infty} \frac{\log \mathcal{N}_{\varepsilon_p}^{(a,0)}(\Gamma g) }{-\log (\vep_p)}	\\
		&\geq 1+\varlimsup_{p\to \infty} \frac{\log \big( O_{r,p}(g,J) +|J| \big) -\log 2 }{ p \log T_r- \log|J|}  
		=1+ \varlimsup_{p\to \infty} \frac{\log \big(   O_{r,p}(g,J) +1\big)}{ p \log T_r}.
	\end{align*}
	Similarly,
	\begin{align*}
		\updimB E 
		&\leq 1+\varlimsup_{p\to \infty} 	 \frac{ \log \big(    O_{r,p}(g,J) +2|J| \big)}{ p \log T_r- \log|J|}  
		=1+ \varlimsup_{p\to \infty} \frac{\log \big(   O_{r,p}(g,J) +1\big)}{ p \log T_r}.
	\end{align*}
	Combining these two inequalities, \eqref{eq:box-dim-cal-tmp-4} holds.
	
	Using similar arguments, we know that \eqref{eq:box-dim-cal-tmp-3} holds.
\end{proof}

The following result is a direct corollary of the above lemma.
\begin{lem}\label{lem:osci-box-dim}
	Assume that $g$ is a continuous function on a closed interval $J$.
	Then for any $1\leq r\leq m$ and $\vep>0$, there exists $c_{\vep}>0$ such that 
	\begin{equation}\label{eq:boxx1}
		c_{\vep}^{-1} {T_r}^{p \big( \lowdimB \Gamma g -1-\vep\big)}
		\leq O_{r,p}(g,J) +1
		\leq c_{\vep} {T_r}^{p \big(  \updimB \Gamma g -1+\vep\big)}, \quad  p \geq 1.
	\end{equation} 
\end{lem}
\begin{proof}	
	For any $1\leq r\leq m$ and $\vep>0$, from Lemma~\ref{lem:box-dim-new-x1}, there exists $P\in \bZp$ such that for all $p>P$,
	$$
	\lowdimB \Gamma g -1-\vep \leq \frac{\log \big(O_{r,p}(g,J)+1\big)}{p\log T_r} \leq \updimB \Gamma g -1+\vep,
	$$
	and therefore
	\begin{equation}\label{eq:lem4.3-1}
	    {T_r}^{p \big( \lowdimB \Gamma g -1-\vep\big)} \leq O_{r,p}(g,J)+1  \leq {T_r}^{p \big( \updimB \Gamma g -1+\vep\big)}, \quad p>P.		
	\end{equation}
	Choose $c_\epsilon>1$ large enough, such that for $p=1,2,\ldots,P$,
	\[
	c_\epsilon^{-1} {T_r}^{p \big( \lowdimB \Gamma g -1-\vep\big)}
	\leq O_{r,p}(g,J) +1
	\leq c_\epsilon {T_r}^{p \big(  \updimB \Gamma g -1+\vep\big)}.
	\]
	Combining this with \eqref{eq:lem4.3-1}, we know that \eqref{eq:boxx1} holds.
\end{proof}


\subsection{A crucial lemma for box dimension estimate}

By the definition of $W_n$, it is easy to see that $W_n(x,y)=(L_n(x),F_n(x,y))$, where
$F_n(x,y)=S_n(x)y+q_n(x)$.
From \eqref{eq:rfif}, 
\[
   W_n(x,f(x))=(L_n(x),f(L_n(x))), \quad x\in D_n. 
\]
Thus, we have the following useful equality:
\begin{equation}\label{eq:RFIF-Rec-Expression}
	f(L_n(x))=S_n(x)f(x)+q_n(x), \quad x\in D_n, \; n=1,2,\ldots,N.
\end{equation}

Write $M_{f,J}=\sup_{x \in J}|f(x)|$ for any $J\subset [x_0,x_N]$.
By using similar arguments as in the proof of Lemma~4.4 in \cite{JR24}, we have the following result. 
\begin{lem}\label{lem:3-osci-1}
	For any $1\leq n \leq N$, $J\subset D_n$ and $x \in J$, 
	\begin{equation}\label{eq:Lem4.4}	
		\Big| O(f,L_n(J)) -|S_n(x)| \cdot O(f,J) \Big|\leq 2 M_{f,J} O(S_n,J)+O(q_n,J).
	\end{equation}
\end{lem}

\begin{proof}
	From \eqref{eq:RFIF-Rec-Expression},
	\[
		O(f,L_n(J))=\sup_{x',x'' \in J}\big|S_n(x')f(x')-S_n(x'')f(x'') + q_n(x')-q_n(x'')\big|.
	\]
	Thus
	\begin{equation}\label{eq:Lem4.4-1}
		-O(q_n,J) \leq O(f,L_n(J))-\sup_{x',x'' \in J}\big|S_n(x')f(x')-S_n(x'')f(x'')\big| \leq O(q_n,J).
	\end{equation}
	
	For any $x',x''\in J$, we have 
	\[
	  \big|S_n(x')f(x')-S_n(x'')f(x'') - S_n(x)\big(f(x')-f(x'')\big)\big| \leq 2M_{f,J}O(S_n,J). 
	\]
	Thus
	\[
	  \sup_{x',x''\in J}\big|S_n(x')f(x')-S_n(x'')f(x'')|\geq |S_n(x)|O(f,J)-2M_{f,J} O(S_n,J) 
	\]
	and
	\[
	  \sup_{x',x''\in J}\big|S_n(x')f(x')-S_n(x'')f(x'')|\leq |S_n(x)|O(f,J)+2M_{f,J} O(S_n,J).
	\]
	Combining these two inequalities with \eqref{eq:Lem4.4-1}, we can easily see that \eqref{eq:Lem4.4} holds.
\end{proof}

The following lemma plays a crucial role to estimate the box dimension of $\Gamma f$.

\begin{lem}\label{lem:box-pass}

	Let $1 \leq n \leq N$ and $k \in \Z^+$. Assume that $J\subset D_n$ is a closed interval. Then the following statements hold.
	\begin{enumerate}
		\item 
		$ \updimB \big(\Gamma f|_{L_n(J)} \big) \leq  \updimB \big(\Gamma f|_{J} \big).$	
		\item If $S_n$ is positive on $J$, then 
		$$
		 \updimB \big(\Gamma f|_{L_n(J)} \big)= \updimB \big(\Gamma f|_J \big)\quad \mbox{and}\quad
		\lowdimB \big(\Gamma f|_{L_n(J)} \big)=\lowdimB \big(\Gamma f|_J \big).
		$$		
	\end{enumerate}
\end{lem}

\begin{proof}
	Write $\ups =\max \big\{ \big| S_n(x) \big|: x \in J \big\}$ and 
	$\lows =\min \big\{ \big| S_n(x) \big|: x \in J  \big\}$. 
	
	From Lemma~\ref{lem:3-osci-1}, for any closed interval $E\subset J$, we have
	\begin{align*}
		O(f,L_n(E))&\leq \max_{x\in E} |S_n(x)| \cdot O(f,E)+ 2 M_{f,E} O(S_n,E)+O(q_n,E)\notag\\
		&\leq 	\ups \cdot O(f,E)+ 2 M_{f,J} \Var(S_n,E)+\Var(q_n,E). 
	\end{align*}
    and	
    \begin{align*}
    	O(f,L_n(E))&\geq \min_{x\in E} |S_n(x)| \cdot O(f,E)- 2 M_{f,E} O(S_n,E)-O(q_n,E)\notag\\
    	&\geq 	\lows \cdot O(f,E)- 2 M_{f,J} \Var(S_n,E)-\Var(q_n,E). 
    \end{align*}
	By these two inequalities, it is straightforward to see that
	\begin{equation}\label{eq:Lem4-3-1} 
	    \lows \cdot  O_{r,p}(f,J)-\xi \leq  O_{r,p}(f,L_n(J))\leq \ups \cdot O_{r,p}(f,J)+ \xi, \quad  p \geq 1,
    \end{equation}
	where $\xi=2 M_{f,J} \Var(S_n,J)+\Var(q_n,J)<\infty$.

	Write $\upd= \updimB \big(\Gamma f|_{J}\big)$ and $\lowd=\lowdimB \big(\Gamma f|_{J}\big)$.
	From Lemma~\ref{lem:osci-box-dim}, for any $\vep>0$, there exists $c_\vep>0$ such that for all $p\in \Z^+$,
	\begin{equation}\label{eq:Lem4-3-2}
		c_\vep^{-1} {T_r}^{p ( \lowd -1-\vep)}
		\leq O_{r,p}(f,J) +1
		\leq c_\vep {T_r}^{p ( \upd -1+\vep )}.
	\end{equation}
	Thus, from Lemma~\ref{lem:box-dim-new-x1} and $\upd -1+\vep>0$,
	\begin{align*}
		\updimB \big(\Gamma f|_{L_n(J)}\big) 
		&\leq 1+\varlimsup_{p \to \infty} \frac{\log \big(\ups O_{r,p}(f,J)+\xi+1\big)}{p \log T_r} \\
		&\leq 1+\varlimsup_{p \to \infty} \frac{\log \big(\ups c_\vep  {T_r}^{p ( \upd -1+\vep)}+\xi+1\big)}{p \log T_r} 
		= \upd+\vep .
	\end{align*}
	From the arbitrariness of $\vep$, we have
	\begin{equation}\label{eq:Lem4-3-3}
		 \updimB \big(\Gamma f|_{L_n(J)}\big) \leq  \upd = \updimB \big(\Gamma f|_{J}\big).
	\end{equation}
	Hence, the first statement of this lemma holds.
	
	From now on, we assume that $S_n$ is positive on $J$. Then $\ups \geq \lows >0$. In the case that $\lowd=1$, it is obvious that    \begin{equation}\label{eq:Lem4-3-4}
		\lowdimB \big(\Gamma f|_{L_n(J)}\big)\geq \lowd. 
	\end{equation}
	
	In the case that $\lowd>1$. For any $0<\vep < \lowd-1$, we choose a positive constant $c_\vep$ such that \eqref{eq:Lem4-3-2} holds for all $p\in \bZp$. It is easy to see that for sufficiently large $p$,
	\[
	\lows c_\vep^{-1}  {T_r}^{p ( \lowd -1-\vep)}-\xi+1>0.
	\]
	Thus, from Lemma~\ref{lem:box-dim-new-x1},
	\begin{align*}
		\lowdimB (\Gamma f|_{L_n(J)}) 
		&\geq 1+\varliminf_{p \to \infty} \frac{\log \big(\lows \cdot O_{r,p}(f,J)-\xi+1\big)}{p \log T_r} \\
		&\geq 1+\varliminf_{p \to \infty} \frac{\log \big(\lows c_\vep^{-1}  {T_r}^{p ( \lowd -1-\vep)}-\xi+1\big)}{p \log T_r} 
		= \lowd-\vep .
	\end{align*}
	From the arbitrariness of $\vep$, the inequality \eqref{eq:Lem4-3-4} still holds. Therefore, this equality holds for all $\lowd\geq 1$. Combining this with \eqref{eq:Lem4-3-3}, we have
	\begin{equation}\label{eq:Lem4-3-5}
		\lowd \leq \lowdimB \big(\Gamma f|_{L_n(J)}\big)\leq  \updimB \big(\Gamma f|_{L_n(J)}\big) \leq  \upd .
	\end{equation}

	From \eqref{eq:Lem4-3-1}, we have
	\[
	\ups^{-1} \big(  O_{r,p}(f, L_n(J)) - \xi \big) \leq
	O_{r,p}(f,J) \leq \lows^{-1}  \big(  O_{r,p}(f,L_n(J)) +\xi \big).
	\]
	Using similar arguments in the proof of \eqref{eq:Lem4-3-5}, we can obtain
	\[
	\lowdimB \big(\Gamma f|_{L_n(J)}\big)\leq\lowd\leq  \upd \leq   \updimB \big(\Gamma f|_{L_n(J)}\big)  .
	\]
	Combining this with \eqref{eq:Lem4-3-5}, the second statement of this lemma holds.
\end{proof}

By replacing $J$ with $D_{r,i}^k$ in the above lemma, we immediately obtain the following corollary.
\begin{cor}\label{cor:box-pass}
	Let $1 \leq r \leq m$, $k \geq 1$ and $1 \leq i \leq d_r {T_r}^{k-1}$. Then the following statements hold.
	\begin{enumerate}
		\item 	$ \updimB \big(\Gamma f|_{I_{r,i}^k} \big) \leq  \updimB \big(\Gamma f|_{D_{r,i}^k} \big)$.
		\item  Let $n$ be the unique element in $\{1,\ldots,N\}$ satisfying $I_{r,i}^k\subset I_n$. 
		Assume that $S_n$ is positive on $D_{r,i}^k$. Then 
		\[
		 \updimB \big(\Gamma f|_{I_{r,i}^k} \big)= \updimB \big(\Gamma f|_{D_{r,i}^k} \big)\quad \mbox{and}\quad
		\lowdimB \big(\Gamma f|_{I_{r,i}^k} \big)=\lowdimB \big(\Gamma f|_{D_{r,i}^k} \big).
		\]			
	\end{enumerate}
\end{cor}

\subsection{Estimate on upper box dimension of $\Gamma f|_{ B_{r,1} \setminus B_{r,k}}$}

Fix $1 \leq r \leq m$. Recall that $\mB_{r,0}=\{D_n:\, n\in \Lambda_{r}\}$, $\mB_{r,1}=\{I_n:\, n\in \Lambda_r\}$, $B_{r,0}=\cup \mB_{r,0}$ and $B_{r,1}=\cup \mB_{r,1}$. In the case that $ B_{r,0}=B_{r,1}$, we define $\uplam_{r,k}=1$ for all $k \in \Z^+$, otherwise, we define
\begin{align*}
	\uplam_{r,1} &=\max\big\{  \updimB \big(\Gamma f|_{I_i}\big): I_i \subset B_{r,0} \textrm{ and } I_i \not\subset B_{r,1} \big\}, \quad \mbox{and} \\
	\uplam_{r,k} &=\max\big\{ \updimB\big(\Gamma f|_{I_{r,i}^k}\big): I_{r,i}^k\not\in \mB_{r,k}   \big\}, \quad k\geq 2.
\end{align*}

From Corollary~\ref{cor:box-pass}, we can obtain the following result.

\begin{lem}\label{lem:box-up-eks}
	Assume that $1 \leq r \leq m$. Then the following statements hold.
	\begin{enumerate}
		\item We have $\uplam_{r,k} \leq \uplam_{r,1}$ for all $k \geq 2$. That is,
		$$
		\updimB \big(\Gamma f|_{ B_{r,1} \setminus B_{r,k} }\big)\leq \uplam_{r,1}.
		$$
		\item If $S_n$ is positive on $D_n$ for all $n\in \Lambda_r$. Then $\uplam_{r,k}=\uplam_{r,1}$ for all $k \geq 2$. We denote the common value by $\uplam_r$.
	\end{enumerate}
\end{lem}

\begin{proof}
	Without loss of generality, we assume that $B_{r,0}\not=B_{r,1}$. 
	
	First, we prove (1). Fix $k\geq 1$. For any $I_{r,j}^{k+1}\not\in \mB_{r,k+1}$, let $n\in \{1,\ldots,N\}$ be the unique element satisfying $I_{r,j}^{k+1}\subset I_n$.
	In the case that $D_{r,j}^{k+1} \not\subset B_{r,1}$, there exists $t \notin \Lambda_r$ such that $D_{r,j}^{k+1} \subset I_t$. From Corollary~\ref{cor:box-pass}, we have
	\[
	\updimB \big(\Gamma f |_{I_{r,j}^{k+1}}\big) \leq \updimB \big(\Gamma f |_{D_{r,j}^{k+1}}\big)
	\leq \updimB \big(\Gamma f |_{I_t}\big) \leq \uplam_{r,1}.
	\]
	
	In the case that $D_{r,j}^{k+1} \subset B_{r,1}$, by using \eqref{eq:Theta-simple-facts-1}, it is easy to see that there exists $i\in\{1,\ldots, d_r{T_r}^{k-1}\}$ such that $I_{r,i}^k=D_{r,j}^{k+1}$. From $L_n(I_{r,i}^k)=I_{r,j}^{k+1}\not\in \mB_{r,k+1}$, we have $I_{r,i}^k\not\in\mB_{r,k}$. Combining this with Corollary~\ref{cor:box-pass}, 
	\[
	\updimB \big(\Gamma f |_{I_{r,j}^{k+1}}\big) \leq \updimB \big(\Gamma f |_{D_{r,j}^{k+1}}\big)= \updimB \big(\Gamma f |_{I_{r,i}^{k}}\big)\leq \uplam_{r,k}.
	\]
	
	From the arbitrariness of $j$, we have $\uplam_{r,k+1}\leq \max\{ \uplam_{r,k},\uplam_{r,1}\}$.
	By the arbitrariness of $k$, we have $\uplam_{r,k} \leq \uplam_{r,1}$ for all $k\geq 2$.
	
	Now we prove (2). Assume that $S_n$ is positive on $D_n$ for all $n\in \Lambda_r$. From (1), it suffices to show that $\uplam_{r,k}\leq \uplam_{r,k+1}$ for all $k\geq 1$. In the case that $k\geq 2$, we have by definition that
	\[
	\uplam_{r,k}=\updimB\big( \Gamma f|_{ B_{r,1} \setminus B_{r,k} }\big).
	\]
	Combining this with the fact that $B_{r,k+1} \subset B_{r,k}$, we obtain that $\uplam_{r,k}\leq \uplam_{r,k+1}$ for all $k\geq 2$.
	
	In the case that $k=1$, for any $i$ with $I_i \subset B_{r,0}$ and $I_i \not\subset B_{r,1}$ , we have $i\not\in \Lambda_r$ and there exists $n\in \Lambda_r$ satisfying $I_i\subset D_n$. Let $j$ be the unique element in $\{1,\ldots, d_rT_r\}$ such that $I_{r,j}^2=L_n(I_i)$. Then $I_i=D_{r,j}^2 $. From $i\not\in \Lambda_r$, we have $I_i\not \in \mB_{r,1}$ so that $I_{r,j}^2=L_n(I_i)\not\in \mB_{r,2}$. 
	Combining this with Corollary~\ref{cor:box-pass}, 
	\begin{equation*}
		\updimB \big(\Gamma f|_{I_i} \big)=\updimB \big(\Gamma f|_{D_{r,j}^2} \big)=\updimB \big(\Gamma f|_{I_{r,j}^2} \big)\leq \uplam_{r,2}.
	\end{equation*}
	From the arbitrariness of $I_i$, $\uplam_{r,1} \leq  \uplam_{r,2}$.
\end{proof}

\section{Box dimension of generalized affine RFIFs}
\subsection{Oscillation sum of generalized affine RFIFs}

Recall that $M_{f,J}=\sup_{x \in J}|f(x)|$ for any $J\subset [x_0,x_N]$. 
By using $\ups_{r,i,j}^{k}$ and $\lows_{r,i,j}^{k}$, we can obtain the following inequalities on the oscillation sum of $f$.

\begin{lem}\label{lem:2-3}
	Let $1\leq r\leq m$, $k \in \Z^+$ and $i\in \Theta_{r,k}$. 	
	Assume that $1 \leq n \leq N$ is the unique element satisfying $I_{r,i}^k \subset I_{n}$.
	Then for any $p \in \mathbb{Z}^+$, we have
	\begin{align}
		&\sum_{i'\in \Theta_{r,k+1}:\, I_{r,i'}^{k+1}\subset I_{r,i}^k} O_{r,p}(f,I_{r,i'}^{k+1}) \leq \sum_{j\in\Theta_{r,k}} \ups_{r,i,j}^k O_{r,p}(f,I_{r,j}^k)+\xi_{r,k,i}, \label{eq:lem2-3-1}\\
		&\sum_{i'\in \Theta_{r,k+1}:\, I_{r,i'}^{k+1}\subset I_{r,i}^k} O_{r,p}(f,I_{r,i'}^{k+1}) \geq \sum_{j\in \Theta_{r,k}} \lows_{r,i,j}^k O_{r,p}(f,I_{r,j}^k)-\xi_{r,k,i}, \label{eq:lem2-3-2}
	\end{align}
	where $\xi_{r,k,i}=2 M_{f,B_{r,k}} \Var (S_{n},D_{r,i}^k)+\Var (q_{n},D_{r,i}^k)$.
\end{lem}

\begin{proof}
	For each $i'\in \Theta_{r,k+1}$ satisfying $I_{r,i'}^{k+1} \subset I_{r,i}^k$, there exists $j\in \Theta_{r,k}$ such that $I_{r,j}^k=D_{r,i'}^{k+1}\in \mB_{r,k}$. 
	Notice that $I_{r,j}^k=D_{r,i'}^{k+1}\subset D_{r,i}^k$. Thus, by letting $J=I_{r,j}^k$ in \eqref{eq:Lem4-3-1}, we have
	\begin{align*}
		O_{r,p}(f,L_n(I_{r,j}^k)) \leq \ups_{r,i,j}^k O_{r,p}(f,I_{r,j}^k)+ 2 M_{f,B_{r,k}} \Var(S_n,I_{r,j}^k)+\Var(q_n,I_{r,j}^k),
	\end{align*}
	where we use the fact that $M_{f,I_{r,j}^k}\leq M_{f,B_{r,k}}$.	Hence,
	\begin{align*}
		&\sum_{i'\in \Theta_{r,k+1}:\, I_{r,i'}^{k+1}\subset I_{r,i}^k} O_{r,p}(f,I_{r,i'}^{k+1}) \\
		= &\sum_{j\in \Theta_{r,k}:\, I_{r,j}^k\subset D_{r,i}^k}O_{r,p}(f,L_n(I_{r,j}^k)) \\
		\leq &\sum_{j\in \Theta_{r,k}:\, I_{r,j}^k\subset D_{r,i}^k} \ups_{r,i,j}^k O_{r,p}(f,I_{r,j}^k)+2 M_{f,B_{r,k}} \Var (S_{n},D_{r,i}^k)+\Var (q_{n}, D_{r,i}^k)\\
		= &\sum_{j\in \Theta_{r,k}} \ups_{r,i,j}^k O_{r,p}(f,I_{r,j}^k)+2 M_{f,B_{r,k}} \Var (S_{n},D_{r,i}^k)+\Var (q_{n}, D_{r,i}^k).
	\end{align*}
	Thus, \eqref{eq:lem2-3-1} holds. Using similar arguments, \eqref{eq:lem2-3-2} holds.
\end{proof}


Now we define $V(f,r,k,p)$ to be the transpose of the vector $\big(O_{r,p}(f,I_{r,i}^k)\big)_{i\in \Theta_{r,k}}$, i.e., 
\[
\big(V(f,r,k,p)\big)_i=O_{r,p}(f,I_{r,i}^k\big), \quad i\in \Theta_{r,k}.
\]
We also define $\xi_{r,k}$ to be the transpose of the vector $\big(\xi_{r,k,i}\big)_{i\in \Theta_{r,k}}$. It is obvious that
\begin{equation}\label{eq:xiup}
	\lVert \xi_{r,k}\rVert_1 = \sum_{i\in\Theta_{r,k}}  \xi_{r,k,i} \leq\sum_{n \in \Lambda_r}\Big(2M_{f,B_{r,k}} \Var (S_n,D_n)+\Var(q_n,D_n)\Big),
\end{equation}
where $\lVert \bfv\rVert_1:=\sum_{i=1}^q |v_i|$ for any $\bfv=(v_1,\ldots,v_q)\in \bR^q$.

Similarly, we define 
\[
  O_{r,p}^*(f,I_{r,i}^k)=\sum_{i':\, I_{r,i'}^{k+1}\subset I_{r,i}^k\setminus (B_{r,k+1})^{\circ}} O_{r,p}(f,I_{r,i'}^{k+1}), \quad i\in \Theta_{r,k},
\]
and let $\widetilde{V}(f,r,k,p)$ be the transposition of the vector $\big(O_{r,p}^*(f,I_{r,i}^k)\big)_{i\in \Theta_{r,k}}$.

From Lemma~\ref{lem:2-3}, the following result holds.
\begin{thm}\label{thm:vsm-pass-tmp}
	Let $1 \leq r \leq m$. For any $k,p \in \mathbb{Z}^+$, we have
	\begin{align}
		V(f,r,k,p+1) \geq &\big(\lowM_{r,k}|_{\Theta_{r,k}}\big) V(f,r,k,p) +\widetilde{V}(f,r,k,p)  -\xi_{r,k},\label{eq:3-vector-pass-1}\\
		V(f,r,k,p+1) \leq &\big(\upM_{r,k}|_{\Theta_{r,k}}\big)  V(f,r,k,p)+ \widetilde{V}(f,r,k,p)+\xi_{r,k}.\label{eq:3-vector-pass-2}
	\end{align}
\end{thm}
\begin{proof}
	It is clear that for any $1\leq r\leq m$, $k,p\in \Z^+$ and any $i\in \Theta_{r,k}$, 
	\[
	O_{r,p+1}(f,I_{r,i}^k)=\sum_{i':\, I_{r,i'}^{k+1} \subset I_{r,i}^k} O_{r,p}(f,I_{r,i'}^{k+1}).
	\]

	Notice that for any $1 \leq i' \leq d_r {T_r}^k$ such that $I_{r,i'}^{k+1} \subset I_{r,i}^k$, we have either $ i' \in \Theta_{r,k+1}$ or $I_{r,i'}^{k+1}\subset I_{r,i}^k\setminus (B_{r,k+1})^\circ$.	Thus, from Lemma~\ref{lem:2-3}, 
	\[
	  O_{r,p+1}(f,I_{r,i}^k)\leq \sum_{j\in\Theta_{r,k}} \ups_{r,i,j}^k O_{r,p}(f,I_{r,j}^k)+\xi_{r,k,i} + O_{r,p}^*(f,I_{r,i}^k),
	\]
	which implies that \eqref{eq:3-vector-pass-2} holds. Similarly, we know that \eqref{eq:3-vector-pass-1} holds.
\end{proof}

\subsection{Upper box dimension of $\Gamma f|_{B_{r,k}}$ }

Using the finite stability of box dimension, we have the following result.
\begin{lem}\label{lem:box-def-3}
	Let $1 \leq r \leq m$. For any $k,p \in \mathbb{Z}^+$, we have
	\begin{align}
		\lowdimB \big(\Gamma f|_{B_{r,k}} \big) & = 1+\varliminf_{p \to \infty} \frac{\log \big(\lVert V(f,r,k,p) \rVert_1+1\big)}{p \log T_r}, \label{eq:vec-dim-pass-1}  \\
		\updimB \big(\Gamma f|_{B_{r,k}} \big)
		&= 1+\varlimsup_{p \to \infty} \frac{\log \big(\lVert V(f,r,k,p) \rVert_1+1\big)}{p \log T_r}. \label{eq:vec-dim-pass-2} 
	\end{align}
\end{lem}
\begin{proof}
  Fix $p\in \Z^+$. For each $i\in \Theta_{r,k}$, by letting $J= I_{r,i}^k$, $g=f|_{I_{r,i}^k}$ and $a=x_0$ in \eqref{eq:boxxx1}, we can obtain
	\begin{equation*}
		\max \big\{ | I_{r,i}^k |, O_{r,p}(f,{I_{r,i}^k})  \big\}
		\leq \vep_{p} \cdot \mathcal{N}_{\vep_{p}}^{(x_0,0)} \big(\Gamma f|_{I_{r,i}^k}\big) 
		\leq     O_{r,p}(f,{I_{r,i}^k}) +2 |{I_{r,i}^k}|,
	\end{equation*}
	where $\vep_{p}=| I_{r,i}^k |/{T_r}^p=(x_N-x_0)N^{-1}{T_r}^{-k-p+1}$.
	Notice that
	\[
	  \mathcal{N}_{\vep_p}^{(x_0,0)} (\Gamma f|_{B_{r,k}}) =\sum_{i \in \Theta_{r,k}}
	  \mathcal{N}_{\vep_p}^{(x_0,0)} (\Gamma f|_{I_{r,i}^k}).
	\]
	Thus, from $\lVert V(f,r,k,p) \rVert_1=\sum_{i \in \Theta_{r,k}} O_{r,p}(f,{I_{r,i}^k})$, 
	\[
    	\max\big\{\beta_{r,k}, \lVert V(f,r,k,p) \rVert_1  \big\}
    	\leq
	    \vep_{p} \cdot\mathcal{N}_{\vep_{p}}^{(x_0,0)} \big(\Gamma f|_{B_{r,k}}\big) 
	    \leq \lVert V(f,r,k,p) \rVert_1+2\beta_{r,k},
	\]
	where $\beta_{r,k}= (x_N-x_0)N^{-1}{T_r}^{-k+1}\card  (\Theta_{r,k})$. 
	Using similar arguments in Lemma~\ref{lem:box-dim-new-x1},
	we know that \eqref{eq:vec-dim-pass-1} and \eqref{eq:vec-dim-pass-2} hold. 
\end{proof}

\begin{lem}\label{lem:vsm-pass-tmp-2}
	Let $1 \leq r \leq m$.  
	Then for any $k\geq 1$ and $\vep>0$, there exists $c_{\vep,r,k}>0$ such that for all $p \in \mathbb{Z}^+$, 
	\begin{equation*}
		V(f,r,k,p+1) \leq (\upM_{r,k}|_{\Theta_{r,k}}) V(f,r,k,p)+c_{\vep,r,k} {T_r}^ {p (\uplam_{r,1}-1+\vep)} \bfe_{r,k},
	\end{equation*}
	where $\bfe_{r,k} = (1,\ldots,1)^{T} \in \mathbb{R}^{\card (\Theta_{r,k})}$.
\end{lem}

\begin{proof}
	From Theorem~\ref{thm:vsm-pass-tmp}, it suffices to show that for any $k\geq 1$ and $\vep>0$, 
	there exists $c_{\vep,k}>0$ such that for all $p \in \mathbb{Z}^+$,
	\begin{equation}\label{eq:Lem4.9-1}
		\widetilde{V}(f,r,k,p) +\xi_{r,k} \leq c_{\vep,k} {T_r}^ {p (\uplam_{r,1}-1+\vep)} \bfe_{r,k} .
	\end{equation}
	
	Recall that for each $i\in \Theta_{r,k}$,
	\[
	\big(\widetilde{V}(f,r,k,p)\big)_i=
	 O_{r,p}^*(f,I_{r,i}^k)=
	\sum_{i':\, I_{r,i'}^{k+1}\subset I_{r,i}^k\setminus (B_{r,k+1})^{\circ}} O_{r,p}(f,I_{r,i'}^{k+1}).
	\]
	From Lemma~\ref{lem:box-up-eks}, 
	for each $1\leq i'\leq d_r {T_r}^{k}$ with $I_{r,i'}^{k+1}\subset I_{r,i}^k$ and $I_{r,i'}^{k+1}\not\in 
	\mB_{k+1}$, 
	\[
	  \updimB \big(\Gamma f|_{I_{r,i'}^{k+1}} \big) 	\leq \uplam_{r,k+1} \leq\uplam_{r,1}.
	\]
	Combining this with Lemma~\ref{lem:osci-box-dim}, there exists $c_{\vep}(r,k,i,i')>0$ such that
	\[
	  O_{r,p}(f,I_{r,i'}^{k+1}) \leq  c_{\vep}(r,k,i,i') \cdot {T_r}^{p ( \uplam_{r,1} -1+\vep)}, \quad \forall p \in \mathbb{Z}^+ .
	\]
	Notice that
	$ 
	\card\big(\{i':\, I_{r,i'}^{k+1}\subset I_{r,i}^k\setminus (B_{r,k+1})^{\circ}\}\big)\leq T_r.
	$
	Thus, we define
	\[
	  c_{\vep,r,k}'=T_r \max_{i \in \Theta_{r,k} } \Big( \max_{i':\, I_{r,i'}^{k+1}\subset I_{r,i}^k\setminus (B_{r,k+1})^{\circ}} c_{\vep}(r,k,i,i')\Big).
	\]
	Then $c_{\vep,r,k}'>0$ and for all $i\in \Theta_{r,k}$ and $p\in \Z^+$,
	\begin{equation}\label{eq:Lem4.9-2}
		\big(\widetilde{V}(f,r,k,p)\big)_i\leq c_{\vep,r,k}' {T_r}^{p ( \uplam_{r,1} -1+\vep)}.
	\end{equation}
	
	By \eqref{eq:xiup}, for each $i\in \Theta_{r,k}$,
	\begin{equation*}
		(\xi_{r,k})_i\leq \lVert \xi_{r,k}\rVert_1 \leq\sum_{n \in \Lambda_r}\Big(2M_{f,B_{r,k}} \Var (S_n,D_n)+\Var(q_n,D_n)\Big)<\infty.
	\end{equation*}
	Combining this with $\uplam_{r,1}-1\geq 0$, we have $(\xi_{r,k})_i\leq \lVert \xi_{r,k}\rVert_1 {T_r}^{p(\uplam_{r,1}-1+\vep)}$. By letting $c_{\vep,r,k}=c_{\vep,r,k}'+\lVert \xi_{r,k}\rVert_1$, we know from this inequality and \eqref{eq:Lem4.9-2} that for all $i\in \Theta_{r,k}$ and $p\in \Z^+$,
	\[
	\big(\widetilde{V}(f,r,k,p)\big)_i + (\xi_{r,k})_i\leq c_{\vep,r,k} {T_r}^{p ( \uplam_{r,1} -1+\vep)},
	\]
	which implies that \eqref{eq:Lem4.9-1} holds.
\end{proof}


In order to estimate the upper box dimension of $\Gamma f|_{B_{r,1}}$, we will use similar method in the proof of \cite[Theorem~4.6]{JR24}.
For every nonnegative $n \times n$ matrix $A=(a_{ij})$, we define 
$$\|A\|_\infty=\max_{1\leq i\leq n} \sum_{j=1}^n a_{ij}.$$
Let $\bfe=(1,\ldots,1)^T\in \bR^{n}$. Then it is clear that 
\begin{equation}\label{eq:4.13}
	A \bfe \leq \|A\|_\infty \bfe.
\end{equation}
From \cite[Example 5.6.5]{HJ13}, $\lVert \cdot \rVert_\infty$ is a matrix norm. Thus
\begin{equation}\label{eq:4.14}
	\|A^q\|_\infty\leq (\|A\|_\infty)^q.
\end{equation}
Furthermore, from \cite[Corollary 5.6.14]{HJ13}, 
\begin{equation}\label{eq:4.15}
	\lim_{p\to\infty} (\rVert A^p\|_\infty)^{1/p}=\rho(A).
\end{equation}



\begin{lem}\label{thm:box-dim-up-sc}
	Let $1 \leq r \leq m$. Then
	\begin{equation}\label{eq:4.16}
		\updimB \big(\Gamma f|_{B_{r,1}}\big) \leq\max\Big\{ 1+ \frac{\log  \uprho_r }{ \log T_r}, \uplam_{r,1} \Big\}.
	\end{equation}
\end{lem}

\begin{proof}
	
	Fix $ 1 \leq r \leq m$.
	In this proof, for simplicity, we denote $\upM_{r,k}|_{\Theta_{r,k}}$ and $V(f,r,k,p)$ by $\upM_{k}$ and $V(k,p)$, respectively. 
	Write $\bfe_{k}  = (1,\ldots,1)^{T} \in \mathbb{R}^{\card (\Theta_{r,k})}$.

	
	
	For any $k\in\Z^+$ and $\vep>0$, from Lemma~\ref{lem:vsm-pass-tmp-2}, there exists $c_{\vep,k}>0$ such that 
	\begin{equation*}
		V(k,t+1) \leq \upM_{k} V(k,t)+c_{\vep,k} {T_r}^ {t (\uplam_{r,1}-1+\vep)} \bfe_{k}, \quad t \in \Z^+.
	\end{equation*}
	Hence, for all $p,q \in \Z^+$,
	\begin{equation}\label{eq:tmp-5-1}
		V(k,(q+1)p) 
		\leq \big(\upM_k\big)^p V(k,qp) +  c_{\vep,k}\sum_{\ell=0}^{p-1}  \big(\upM_k\big)^{p-1-\ell}   
		 {T_r}^ {(qp+\ell) (\uplam_{r,1}-1+\vep)} \bfe_{k}     .
	\end{equation}
	Fix $p \in \Z^+$.
	Choose $\alpha_{k,p}>0$ big enough such that
	$V(k,p) \leq \alpha_{k,p} \bfe_k$ and 
	\[
	  c_{\vep,k} \sum_{\ell=0}^{p-1}\big(\upM_k\big)^{p-1-\ell}    {T_r}^ {\ell (\uplam_{r,1}-1+\vep)} \bfe_{k} \leq \alpha_{k,p} \bfe_k.
	\]
	Let $\beta={T_r}^ {(\uplam_{r,1}-1+\vep)}$.
	Then $\beta \geq {T_r}^{\vep} >1$. From \eqref{eq:tmp-5-1},
	\begin{equation*}
		V(k,(q+1)p) 
		\leq  \big(\upM_k\big)^p V(k,qp) + {\beta}^{pq} \alpha_{k,p} \bfe_{k}     .
	\end{equation*}
	By induction and using $V(k,p) \leq \alpha_{k,p} \bfe_k$,
	\[
	   V(k,(q+1)p) \leq  \alpha_{k,p} \sum_{\ell=0}^{q} {\beta}^{p(q-\ell)}  \big(\upM_k\big)^{p\ell}   \bfe_{k}.
	\]
	Notice that from \eqref{eq:4.13} and \eqref{eq:4.14}, 
	$$(\upM_k)^{p\ell} \bfe_k\leq \|(\upM_k)^{p\ell}\|_\infty \bfe_k\leq \big(\| (\upM_k)^p\|_\infty\big)^\ell \bfe_k.$$ 
	Combining this with the fact that for all $a>1$, $b\geq 0$ and $0\leq \ell\leq q$,
	\[
	a^{q-\ell}b^\ell \leq \max\{a^q,b^q\}\leq a^{q+1}+ b^{q+1},  
	\]
	we have 
	\[
	V(k,(q+1)p) \leq (q+1) \alpha_{k,p}\Big( \big( \lVert(\upM_k)^{p}\rVert_\infty \big)^{q+1} + {\beta}^{p(q+1)} \Big)  \bfe_{k} .
	\]
	Thus, 
	$
	\lVert V(k,qp) \rVert_1 
	\leq q \alpha_{k,p} \big( \big( \lVert(\upM_k)^{p}\rVert_\infty \big)^q + {\beta}^{pq} \big)  \lVert  \bfe_{k} \rVert_1
	$
	for all $q \in \Z^+$.
	Hence, 
	\begin{align*}
		& \varlimsup_{q \to \infty} \frac{\log \big( \lVert V(k,pq) \rVert_1 +1\big)}{pq \log T_r} \\
		\leq & \varlimsup_{q \to \infty} \frac{\log \Big( q \alpha_{k,p}\Big( \big( \lVert(\upM_k)^{p}\rVert_\infty \big)^q + {\beta}^{pq} \Big)  \lVert  \bfe_{k} \rVert_1+1\Big)}{p q\log T_r} \\
		= &  \frac{\log \big( \max\big\{   \lVert (\upM_k)^{p}\rVert_\infty  , \beta^p \big\}\big)}{p \log T_r} \\
		= & \max\Big\{  \frac{\log   \big( \lVert (\upM_k)^{p} \rVert_\infty \big) }{p \log T_r}, \uplam_{r,1}-1+\vep \Big\}.
	\end{align*}
	Thus, from Lemma~\ref{lem:box-def-3}, we have
	\begin{align*}
		\updimB \big(\Gamma f|_{B_{r,k}}\big) 	
		&\leq 1+\varlimsup_{q \to \infty} \frac{\log \big( \lVert V(k,pq) \rVert_1 +1\big)}{pq \log T_r} \\
		&\leq \max\Big\{ 1+ \frac{\log   \big( \lVert(\upM_k)^{p}\rVert_\infty \big)  }{p \log T_r}, \uplam_{r,1}+\vep \Big\}.
	\end{align*}
	
	Notice that $\lim_{p\to\infty} (\rVert (\upM_k)^p\|_\infty)^{1/p}=\rho(\upM_k)$ (see \eqref{eq:4.15}). Thus, from the arbitrariness of $p$ and $\vep$,
	\begin{equation*}
		\updimB\big(\Gamma f|_{B_{r,k}}\big) 	
		\leq\max\Big\{  1+ \frac{\log  \rho(\upM_{k})  }{ \log T_r},  \uplam_{r,1}\Big\}.
	\end{equation*}
	
	
	From Lemma~\ref{lem:box-up-eks},
	$
	\updimB\big(\Gamma f|_{ B_{r,1} \setminus B_{r,k}}\big) 		\leq \uplam_{r,1}
	$
	for all $k \in \Z^+$. Thus, 
	\[
	\updimB \big(\Gamma f|_{B_{r,1} }\big) 	
	\leq\max\Big\{  1+ \frac{\log  \rho(\upM_{k})  }{ \log T_r},  \uplam_{r,1}\Big\}, \quad  k \in \Z^+.
	\]
	By letting $k$ tend to infinity, we know from Theorem~\ref{thm:rho-up-sub} that \eqref{eq:4.16} holds.
\end{proof}

\subsection{Lower box dimension of $\Gamma f|_{B_{r,1}}$ }

\begin{lem}\label{lem:var-infty}
	Assume that $1 \leq r \leq m$, $k_1\geq 1$ and $S_n$ is positive on $D_n \cap B_{r,k_1}$ for all $n\in \Lambda_r$. 
	 If there exists a closed interval $J \subset B_{r,k_1}$ such that $\Var(f,J)=\infty$,
	then $\Var(f,I_{r,i}^k)=\infty$ for all $k\geq 1$ and all $i \in \Theta_{r,k}$.
\end{lem}

\begin{proof}
	Define $\lows^*_{r}=\min\{ |S_n(x)|:\, x\in D_n\cap B_{r,k_1}, n \in \Lambda_r \}$.
	Since $S_n$ is positive on $D_n \cap B_{r,k_1}$ for all $n\in \Lambda_r$, we have $\lows^*_{r}>0$.
	
	First we claim that $\Var(f,I_{r,i}^k)=\infty$ for all $k> k_1$ and all $i \in \Theta_{r,k}$.
	We prove this by contradiction. Assume that there exist $k> k_1$ and $i\in \Theta_{r,k}$ such that $\Var(f,I_{r,i}^k)<\infty$. Notice that $D_{r,i}^k\in \mathcal{B}_{r,k-1}$. Thus 
	\begin{equation}\label{eq:SnLowBdd}
	    \min\{|S_n(x)|:\, x\in D_{r,i}^k\}\geq s_r^*.
	\end{equation}
	
	Let $n\in \Lambda_r$ satisfy $I_{r,i}^k \subset I_n$. From \eqref{eq:Lem4-3-1} and \eqref{eq:SnLowBdd}, 
	\[
	O_{r,p}(f,D_{r,i}^k) \leq (\lows_{r}^{*})^{-1} \big( O_{r,p}(f, I_{r,i}^k) +\xi_{r,k,i} \big) ,
	\]
	where $\xi_{r,k,i}\leq 2M_{f,[x_0,x_N]} \Var (S_n,D_{r,i}^k ) +\Var(q_n,D_{r,i}^k)<\infty$. Hence, by letting $p$ tend to infinity, we know from 
	Lemma~\ref{lem:VarOinfty} that 
	\[ 
	\Var(f,D_{r,i}^k) \leq  (\lows_{r}^{*})^{-1} \big( \Var(f, I_{r,i}^k) +\xi_{r,k,i}\big)<\infty.
	\]
	As a result, for all $ \ell \in \Theta_{r,k} $ with $I_{r,\ell}^k \subset D_{r,i}^k$, we have
	$$ \Var(f,I_{r,\ell}^k) \leq \Var(f,D_{r,i}^k) <\infty.$$
	
	From Lemma~\ref{lem3.3}, for any $ j \in \Theta_{r,k} $, there exists a finite sequence $\{i_t\}_{t=0}^p$ in $\Lambda_r$, such that $i_0=j$, $i_p=i$, and
    $I_{r,i_{t-1}}^k \subset D_{r,i_{t}}^k$ for all $1\leq t \leq p$.
    It follows from $i_p=i$ that $\Var(f,I_{r,i_p}^k)<\infty$, which gives $\Var(f,D_{r,i_p}^k)<\infty$. Combining this with $I_{r,i_{p-1}}^k \subset D_{r,i_{p}}^k$, we have $\Var(f,I_{r,i_{p-1}}^k)<\infty$ so that $\Var(f,D_{r,i_{p-1}}^k)<\infty$. Continuing this process, we can obtain $\Var(f,I_{r,i_{t}}^k)<\infty$ and $\Var(f,D_{r,i_{t}}^k)<\infty$ for all $0\leq t\leq p$. Thus $\Var(f,I_{r,j}^k)<\infty$ since $i_0=j$.
    
    From the arbitrariness of $j$, we have $ \Var(f,I_{r,j}^k) <\infty$ and $ \Var(f,D_{r,j}^k) <\infty$ for all $j \in \Theta_{r,k}$.
    Notice that
    \[
    \big\{  D_{r,j}^k:\, I_{r,j}^k \in \mB_{r,k}\big\} = \mB_{r,k-1}.
    \]
    Thus, $ \Var(f,I_{r,j}^{k-1}) <\infty$ for all $j \in \Theta_{r,k-1}$ whenever $k> k_1$. 
    By induction, $\Var(f,I_{r,j}^{k_1})<\infty$ for all $j\in \Theta_{r,k_1}$. Combining this with $J\subset B_{r,k_1}=\bigcup_{j\in \Theta_{r,k_1}}I_{r,j}^{k_1}$, we have $\Var(f,J)<\infty$, which is a contradiction. 
    
    
    For any $1\leq k\leq k_1$ and $i\in \Theta_{r,k}$, by Lemma~\ref{lem:Theta-simple-facts}(4), there exists $i'\in \Theta_{r,k_1+1}$ such that $I_{r,i'}^{k+1}\subset I_{r,i}^k$. Thus, by the claim, $\Var(f,I_{r,i}^k)\geq \Var(f,I_{r,i'}^{k'})=+\infty$.
\end{proof}

\begin{lem}\label{thm:box-low-con-1}
	Let $1 \leq r \leq m$ and $k_1\in\bZp$. Assume that $S_n$ is positive on $D_n \cap B_{r,k_1}$ for all $n\in \Lambda_r$, and $\Var(f,E)=\infty$ for some $E \in \mathcal{B}_{r,k_1}$. Then 
	\begin{equation}\label{eq:lem4.13-1}
		\lowdimB \big( \Gamma f|_{B_{r,1}}\big) 
		\geq 1+ \frac{\log \uprho_r}{ \log T_r } .
	\end{equation}
\end{lem}

\begin{proof}
	
	Fix $ 1 \leq r \leq m$.
	In this proof, for simplicity, we denote $\lowM_{r,k}|_{\Theta_{r,k}}$, $V(f,r,k,p)$ and $\xi_{{r,k}}$ by $\lowM_{k}$, $V(k,p)$ and $\xi_{k}$, respectively. 
	Notice that $\{B_{r,k}\}$ is decreasing on $k$ and $S_n$ is positive on $D_n \cap B_{r,k_1}$ for all $n\in \Lambda_r$. Thus $S_n$ is positive on $D_n \cap B_{r,k}$ for all $k\geq k_1$ and all $n\in \Lambda_r$. Hence, from the proof of Theorem~\ref{thm:irr}(2), $\lowM_{k}$ is irreducible for all $k > k_1$. 
	
	Now we prove \eqref{eq:lem4.13-1}. Without loss of generality, we may assume that $\uprho_r>1$. From Theorem~\ref{thm:rho-up-sub}, there exists a positive integer $k_2\geq k_1$, such that $\rho(\lowM_k)>1$ for all $k>k_2$. Fix $k >k_2$. 
	From Theorem~\ref{thm:vsm-pass-tmp}, for all $p\geq 1$,
	\begin{equation}\label{eq:vsm-pass-tmp-6}
		V(k,p+1) \geq \lowM_{k}V(k,p) + \widetilde{V}(f,r,k,p) -\xi_{k}\geq \lowM_{k}V(k,p) -\xi_{k}.
	\end{equation}

	Given $1 <\tau <\rho(\lowM_k)$,
	by Perron-Frobenius Theorem, we can choose a positive eigenvector $\bfw_{k}$ of $\lowM_{k}$ such that $ \lowM_{k} \bfw_{k}=\rho(\lowM_{k})\bfw_{k}$ and
	$$
	\bfw_{k} \geq \frac{1}{\rho(\lowM_{k})-\tau}\xi_{k}. 
	$$
	From Lemma~\ref{lem:var-infty}, for all $i \in \Theta_{r,k}$, we have $\lim_{p \to \infty}O_{r,p}(f,I_{r,i}^k)=\Var(f,I_{r,i}^k)=\infty$.
	Thus, there exists $p_0\in \mathbb{Z}^+$ such that $V(k, p_0) \geq \bfw_{k}$.

	From \eqref{eq:vsm-pass-tmp-6}, 
	$$
	V(k, p_0+1)  \geq \rho(\lowM_{k})  \bfw_{k}-\xi_{k}
	\geq \rho(\lowM_{k})  \bfw_{k}-\big(\rho(\lowM_{k})-\tau\big)\bfw_{k} =\tau \bfw_{k}.
	$$
	Notice that for all $p \in \mathbb{Z}^+$,
	\begin{align*}
		\rho(\lowM_{k}) \tau^p \bfw_{k} -\xi_{k} &=\rho(\lowM_{k})(\tau^p-1)\bfw_{k} +\rho(\lowM_{k})\bfw_{k}-\xi_{k} \\
		&\geq \tau (\tau^p-1) \bfw_{k} +\tau \bfw_{k}=\tau^{p+1} \bfw_{k}.
	\end{align*}
	Thus, by induction, $V(k, p_0+p) \geq \tau^p \bfw_k$ for all $p \in \mathbb{Z}^+ $. Hence, from Lemma~\ref{lem:box-def-3}, 
	\begin{align*}
		\lowdimB \big( \Gamma f|_{B_{r,k}}\big) 
		&\geq 1+ \varliminf_{p \to \infty} \frac{\log \big(\lVert V(k,p_0+p) \rVert_1 +1\big)}{(p_{0}+p) \log T_r } \\
		&\geq 1+\varliminf_{p \to \infty} \frac{\log\big(\tau^p \lVert \bfw_{k} \rVert_1 +1\big)}{(p_{0}+p)\log T_r } 
		= 1+ \frac{\log \tau}{ \log T_r } .
	\end{align*}
	From the arbitrariness of $\tau$, we have
	\[
	\lowdimB \big( \Gamma f|_{B_{r,1}}\big) 
	\geq\lowdimB \big( \Gamma f|_{B_{r,k}}\big) 
	\geq 1+ \frac{\log \rho(\lowM_{k})}{ \log T_r }.
	\]
	By letting $k$ tend to infinity, we know from Theorem~\ref{thm:rho-up-sub} that \eqref{eq:lem4.13-1} holds.
\end{proof}

\subsection{Box dimension of RFIFs}
For $1 \leq i,j \leq N$, we denote by $i\sim j$ if they belong to the same strongly connected component of $\{1,\ldots,N\}$. Otherwise, $i\not\sim j$.
Let
 $$\mA(i) = \{j : j \not\sim i \mbox{ and there is a path in $\{1,\dots,N\}$ from $j$ to $i$} \}.$$
Similarly as in \cite{RXY21}, we define the position $P(i)$, $i=1,2,\ldots,N$ recursively as follows: $P(i)=1$ if $\mA(i)=\emptyset$, otherwise $P(i)=1+\max\{P(j):\, j\in \mA(i)\}$.
 Clearly, $P(i) \leq N$ for all $i$, and $P(i) = P(j)$ if $i \sim j$.

\begin{thm}\label{thm:box-up-final}
We have
\[
	\updimB \Gamma f \leq 1+\max\Big\{  \frac{\log\uprho_1}{\log T_1},\frac{\log\uprho_2}{\log T_2},\ldots, \frac{\log \uprho_m}{ \log T_m},0 \Big\}.
\]
\end{thm}

\begin{proof}
Write $d^* =1 + \max\Big\{  \frac{\log\uprho_1}{\log T_1},\ldots, \frac{\log \uprho_m}{ \log T_m},0 \Big\}$.  
Since $\updimB \Gamma f =\max_{1\leq i\leq N}\updimB \big(\Gamma f|_{I_{i}}\big),$
 it suffices to show that $ \updimB \big(\Gamma f|_{I_i}\big) \leq d^*$ for all $1\leq i\leq N$. We prove this by induction on the position of $i$.

Assume that $P(i)=1$. Then $\mA(i)=\emptyset$. From this fact, it is easy to see that
$$\big\{j\in \{1,\ldots,N\}:\, \mbox{there is a path in $\{1,\ldots,N\}$ from $j$ to $i$} \big\}$$
is a strongly connected component $\Lambda_r$ of $\{1,\ldots,N\}$ for some $1\leq r\leq m$ and $i\in \Lambda_r$. 
Furthermore, we have $\bigcup_{n\in \Lambda_r} D_n=\bigcup_{n\in \Lambda_r}I_n$. Thus $\uplam_{r,1}=1$. From Lemma~\ref{thm:box-dim-up-sc},
$$\updimB \big(\Gamma f|_{I_i}\big)\leq \updimB \big(\Gamma f|_{B_{r,1}}\big) \leq \max\Big\{ 1+ \frac{\log    \uprho_r   }{ \log T_r}, \uplam_{r,1}\Big\}\leq d^*. $$

Fix $1\leq \ell<\max\{P(i):\, 1\leq i\leq N\}$. Assume that $\updimB \big(\Gamma f|_{I_j} \big) \leq d^*$ for all $1\leq j \leq N$ with $P(j)\leq \ell$. Let $1\leq i\leq N$ such that $P(i)= \ell+1$. 

Firstly, we assume that $i$ is not a vertex of any strongly connected component.
Then
$i'\not\sim i$ for all $i'\in\{1,\ldots,N\}$.
For any $1\leq j \leq N$ with $I_{j}\subset D_{i}$, it is clear that there exists a path in $\{1,\ldots,N\}$ from $j$ to $i$. Combining this with $j\not\sim i$, we have $j\in \mA(i) $, which implies that 
$P(j)\leq \ell$. From Lemma~\ref{lem:box-pass} and inductive assumption,
\[ \updimB \big(\Gamma f|_{I_i}\big)  \leq  \updimB \big(\Gamma f|_{D_i}\big) =\max\big\{\updimB \big(\Gamma f|_{I_j}\big):\,I_{j}\subset D_{i}\big\}\leq d^*. \]

Now we assume that $ i \in \Lambda_r$ for some $1\leq r\leq m $. From Lemma~\ref{thm:box-dim-up-sc},
$$\updimB \big(\Gamma f|_{I_i}\big)\leq \updimB \big(\Gamma f|_{B_{r,1}}\big) \leq \max\Big\{ 1+ \frac{\log    \uprho_r   }{ \log T_r}, \uplam_{r,1}\Big\}. $$ 
It follows from $P(i)\geq 2$ that $\bigcup_{n\in \Lambda_r} I_n \subsetneq\bigcup_{n\in \Lambda_r} D_n$, which gives
\begin{align*}
	\uplam_{r,1} =\max\Big\{ \updimB \big(\Gamma f|_{I_n}\big): I_n\subset B_{r,0} \textrm{ and } I_n\not\subset B_{r,1} \Big\}. 
\end{align*}
Notice that $P(n)\leq \ell$ for all $n$ with $I_n\subset B_{r,0}$ and $I_n\not\subset B_{r,1}$.
Hence, by inductive assumption, $\uplam_{r,1}\leq d^*$. As a result, $\updimB \big(\Gamma f|_{I_i}\big) \leq d^*$.
\end{proof}

In the case that $S_n$ is positive on $D_n \cap K_r$ for all $n \in \Lambda_r$, we can obtain better estimate on the upper and lower box dimensions of $\Gamma f|_{B_{r,1}}$ than Lemmas~\ref{thm:box-dim-up-sc} and \ref{thm:box-low-con-1}.
\begin{lem}\label{lem:6.1}
	Let $1 \leq r \leq m$ and $S_n$ is positive on $D_n \cap K_r$ for all $n \in \Lambda_r$. Then
	\begin{equation*}
		d_r^* \leq
		\lowdimB \big(\Gamma f|_{B_{r,1}} \big)  \leq
		\updimB \big(\Gamma f|_{B_{r,1}} \big)  \leq\max \{ d_r^*, \uplam_{r,1}  \},
	\end{equation*}
	where $d_r^*$ is defined by \eqref{eq:def-drstar}. 
\end{lem}

\begin{proof}
	It is clear that $	\lowdimB \big(\Gamma f|_{B_{r,1}} \big)  \geq 1$.
	Let $k_r^* \in \Z^+$ be the smallest element such that $S_n$ is positive on $D_n \cap B_{r,k_r^*}$ for all $n \in \Lambda_r$.
	
	In the case that $\Var(f,E)=\infty$ for some $E \in \mB_{r,k_r^*}$, we have by definition that $d_r^*=1+\log\uprho_r / \log T_r$.
	Thus, the lemma holds from Lemmas~\ref{thm:box-dim-up-sc} and \ref{thm:box-low-con-1}.
	
	In the case that $\Var(f,E)< \infty$ for all $E \in \mB_{r,k_r^*}$, we have by definition that $d_r^*=1$. Thus $\lowdimB \big(\Gamma f|_{B_{r,1}}\big)\geq 1=d_r^*$. On the other hand, by Lemmas~\ref{lem:VarOinfty} and \ref{lem:box-dim-new-x1}, $\dim_B \big(\Gamma f|_E\big)=1$ for all $E\in \mB_{r,k_r^*}$ so that $\dim_B \big(\Gamma f|_{B_{r,k_r^*}}\big)=1$.
	From Lemma~\ref{lem:box-up-eks}, we have
	$
	\updimB\big(\Gamma f|_{ B_{r,1} \setminus B_{r,k}}\big) 		\leq \uplam_{r,1}
	$
	for all $k \in \Z^+$. Thus, 
	$
	\updimB\big(\Gamma f|_{ B_{r,1}} \big) 		\leq \max \{ 1,\uplam_{r,1} \}.
	$
\end{proof}

\begin{proof}[Proof of Theorem~\ref{thm:Main-Results-x1}]
From Theorem~\ref{thm:box-up-final}, it suffices to prove the second part of the theorem. Assume that $S_n $ is positive on $D_n \cap K_r$ for all $1\leq r\leq m$ and all $n\in \Lambda_r$. Write $d_0^*=\max \{ d_1^*, d_2^*, \ldots, d_m^*,1\}.$
It directly follow from Lemma~\ref{lem:6.1} and $\lowdimB \Gamma f \geq 1$ that 
$\lowdimB \Gamma f \geq d_0^*$.

By using similar arguments in Theorem~\ref{thm:box-up-final}, we can show that $\updimB \Gamma f \leq d_0^*$. We sketch the proof for completeness.

In the case that $P(i)=1$, there exists $1\leq r\leq m$ such that $i\in \Lambda_r$ and $\uplam_{r,1}=1$. From Lemma~\ref{lem:6.1},
$$\updimB \big(\Gamma f|_{I_i}\big)\leq \updimB \big(\Gamma f|_{B_{r,1}}\big) \leq \max \{ d_r^*, \uplam_{r,1}\}=\max \{ d_r^*, 1\}\leq d_0^*. $$

Fix $1\leq \ell<\max\{P(i):\, 1\leq i\leq N\}$. Assume that $\updimB \big(\Gamma f|_{I_j}\big) \leq d_0^*$ for all $1\leq j \leq N$ with $P(j)\leq \ell$. Let $1\leq i\leq N$ such that $P(i)= \ell+1$. 

Firstly, we assume that $i$ is not a vertex of any strongly connected component. 
Then from Lemma~\ref{lem:box-pass} and inductive assumption,
\[ \updimB \big(\Gamma f|_{I_i}\big)  \leq  \updimB \big(\Gamma f|_{D_i}\big) =\max\big\{\updimB \big(\Gamma f|_{I_j}\big):\,I_{j}\subset D_{i}\big\}\leq d_0^*. \]

Now we assume that $ i \in \Lambda_r$ for some $1\leq r\leq m $. 
It follows from $P(i)\geq 2$ and inductive assumption that
\begin{align*}
	\uplam_{r,1} =\max\Big\{ \updimB \big(\Gamma f|_{I_n}\big): I_n\subset B_{r,0} \textrm{ and } I_n\not\subset B_{r,1} \Big\}\leq d_0^*. 
\end{align*}
Thus, from Lemma~\ref{lem:6.1},
$$\updimB \big(\Gamma f|_{I_i}\big)\leq \updimB \big(\Gamma f|_{B_{r,1}}\big) \leq \max \{ d_r^*, \uplam_{r,1}\}\leq d_0^*. $$ 
This completes the proof of the theorem.
\end{proof}

\section{An example}

	In this section, we present an example to illustrate our result.
	
	Let $N=6$, $(x_5,y_5)=(5/6,0)$ and $(x_n,y_n)=(n/6,1)$ for all other $n\in\{0,1,\ldots,6\}$. Set 
	\begin{equation}\label{eq:5.1}
		D_1=D_4=[1/3,5/6], \quad D_2=D_3=[1/6,2/3], \quad  D_5=D_6=[2/3,1].
	\end{equation}
	For each $1\leq n\leq 6$, we define functions $L_n(x)$, $S_n(x)$ and $q_n(x)$ on $D_n$ as follows.
	\begin{align*}
		&L_1(x)=\frac{1}{3}x-\frac{1}{9},      && S_1(x)=\frac{1}{2},	&&  q_1(x)=x+\frac{1}{6}, \\ 
		&L_2(x)=\frac{1}{3}x+\frac{1}{9},      && S_2(x)=-\frac{2}{5}x+\frac{49}{60},	&&  q_2(x)=\frac{2}{5}x+\frac{11}{60}, \\ 
		&L_3(x)=\frac{1}{3}x+\frac{5}{18},  &&  S_3(x)=-\frac{1}{5}x+\frac{47}{60},	&& q_3(x)=\frac{1}{5}x+\frac{13}{60}, \\
		&L_4(x)=\frac{1}{3}x+\frac{7}{18}, && S_4(x)=\frac{3}{5}x+\frac{2}{5},  && q_4(x)=\frac{6}{5}x, \\
		&L_5(x)=\frac{1}{2}x+\frac{1}{3}, && S_5(x)=\frac{2}{5}x+\frac{13}{30},  && q_5(x)=-\frac{17}{5}x+\frac{77}{30},  \\
		&L_6(x)=-\frac{1}{2}x+\frac{4}{3}, && S_6(x)=-\frac{1}{2}, &&q_6(x)=-3x+\frac{7}{2}.
	\end{align*}
	Then for any $1 \leq n \leq 6$, we set
	$$
	W_n(x,y)=\big(L_n(x),S_n(x)y +q_n(x)\big), \quad (x,y) \in D_n \times \R.
	$$
	Let $f\in C([0,1])$ be the unique function determined by the recurrent IFS $\{W_1,\ldots,W_6\}$. See Figure~\ref{fig:1}.

	\begin{figure}[htbp]
		\centering
		\includegraphics[scale=0.4]{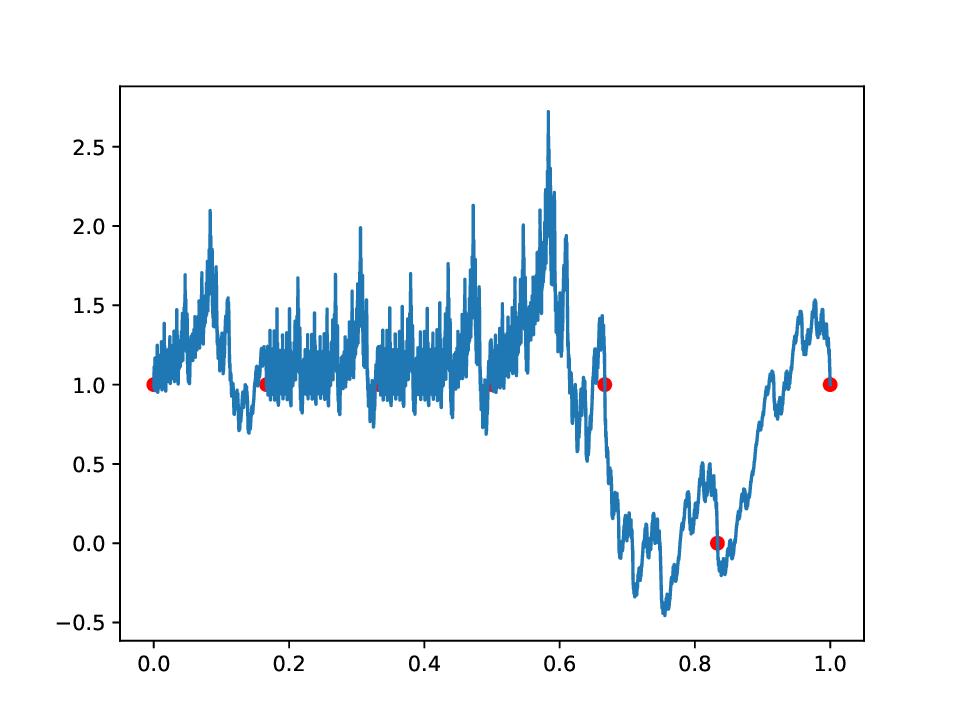}
		\caption{Graph of $f$.}
		\label{fig:1}
	\end{figure}

	From \eqref{eq:5.1}, there are two strongly connected
	components $\Lambda_1,\Lambda_2$ of $\{1,\ldots,6\}$ with $\Lambda_1 =\{2,3,4\}$ and $\Lambda_2 =\{5,6\}$.

	We discuss the component $\Lambda_2$ firstly. It is clear that $\bigcup_{n\in \Lambda_2} D_n=\bigcup_{n\in \Lambda_2} I_n=[2/3,1]$. From Lemma~\ref{lem:equivalent-cr}, $ B_{2,k}=B_{2,1}=[2/3,1]$ for all $k \in \Z^+$. 
	Now we check that $\Var(f,B_{2,1})=\infty$.
	It is clear that $\wdt{V}(f,2,k,p)$ is a zero vector for all $k , \, p \in \mathbb{Z}^+$.	
	By definition, in the case that $k=2$, 
	\[
	\lowM_{2,2}|_{\Theta_{2,2}}=\lowM_{2,2}=
	\begin{pmatrix}
		7/10 & 11 /15  & 0 & 0 \\
		0 & 0 & 23/30  & 4/5 \\
		0 & 0 & 1/2 & 1/2 \\
		1/2 & 1/2 & 0 & 0 
	\end{pmatrix}.
	\]
	We observe that every column sum of $\lowM_{2,2}$ is bigger than $6/5$.	
	From Theorem~\ref{thm:vsm-pass-tmp},
	for any $p \in \mathbb{Z}^+$, we have
	\[
	V(f,2,2,p+1) 
	\geq \big(\lowM_{2,2}\big) V(f,2,2,p)  -\xi_{2,2}.
	\]
		Notice that $\lVert V(f,2,2,p)\rVert_1=O_{2,p+2}(f,B_{2,1})$ for all $ p \in \Z^+$.	Thus,
	\begin{equation}\label{eq:exam-pass}
		O_{2,p+3}(f,B_{2,1}) \geq \frac{6}{5} O_{2,p+2}(f,B_{2,1})  -||\xi_{2,2}||_1, \quad  p \in \Z^+.
	\end{equation}

	Now we are going to estimate $||\xi_{2,2}||_1$.
	Let $\alpha$ be the maximum of $|f|$ on $B_{2,1}$.
	From \eqref{eq:RFIF-Rec-Expression}, we have the following two inequalities,
	\begin{align*}
	\sup\limits_{x \in I_5}|f(x)|  
	&\leq \alpha \cdot \sup\limits_{x \in D_5} |S_5(x)| +  \sup\limits_{x \in D_5 } |q_5(x)| = \frac{5}{6} \alpha+\frac{5}{6}, \\
	\sup\limits_{x \in I_6}|f(x)| 
	&\leq \alpha \cdot \sup\limits_{x \in D_6 } |S_6(x)| + \sup\limits_{x \in D_6 } |q_6(x)| = \frac{1}{2} \alpha +\frac{3}{2}.
	\end{align*}
	Hence, 
	$$
		\alpha =\sup\limits_{x \in I_5 \cup I_6}|f(x)|   \leq  \max \Big\{\frac{5}{6} \alpha+\frac{5}{6}, \frac{1}{2} \alpha +\frac{3}{2} \Big\}.
	$$
	Thus, $\alpha \leq 5$. Furthermore, by simple calculation, 
	\[
	\Var(S_5,D_5)={2}/{15}, \; \Var(q_5,D_5)={17}/{15}, \; \Var(S_6,D_6)=0, \; \Var(q_6,D_6)=1.
	\]
	Thus, from \eqref{eq:xiup} and noticing that $M_{f,B_{2,2}}\leq M_{f,B_{2,1}}$,
	\begin{align*}
		\lVert\xi_{2,2}\rVert_1
		\leq \sum_{n \in \{5,6\}} \Big(2 \alpha \cdot\Var(S_n,D_n) +\Var(q_n,D_n) \Big) \leq \frac{52}{15}.
	\end{align*}

	Combining this with \eqref{eq:exam-pass}, for all $p\in \Z^+$,
	\[
	O_{2,p+3}(f,B_{2,1}) \geq   \frac{6}{5} O_{2,p+2}(f,B_{2,1})-\frac{52}{15}, 
	\]
	and therefore,
	\[O_{2,p+3}(f,B_{2,1})-\frac{52}{3} > \frac{6}{5} \Big( O_{2,p+2}(f,B_{2,1})-\frac{52}{3} \Big).\]
	By direct calculation, $O_{2,10}(f,B_{2,1}) > {52}/{3}$. Thus,
	$$\Var(f,B_{2,1})=\lim_{p \to \infty} O_{2,p} (f,B_{2,1})=\infty.$$
	From Lemma~\ref{lem:var-infty}, $\Var(f,I_5)=+\infty$ so that $d_2^*=1+\log \rho_2/\log T_2$.

	By letting $n=4$ and $J=I_5=[2/3,5/6]$ in \eqref{eq:Lem4-3-1}, we have
	\[
	    O_{r,p}(f,L_4(J)) \geq  \min_{x \in J} |S_4(x)| \cdot  O_{r,p}(f,J)-\xi   , \quad  p \geq 1,
	\]
	where $\xi = 2M_{f,J} \Var(S_4,J)+\Var(q_4,J)<\infty$.
	As a result, 
	\[
	    \Var(f,I_4) \geq \Var(f,L_4(J)) =\lim_{p \to \infty} O_{r,p}(f,L_4(J)) \geq \min_{x \in J} |S_4(x)|  \cdot  \Var(f,J)-\xi  =\infty.
	\]
    Thus, $d_1^*=1+\log \rho_1/\log T_1$.
	
	Using numerical calculation, we obtain Table~\ref{tab:example}.
	Hence, $\rho_1 \approx 1.800$ and $\rho_2 \approx 1.260$.
	Notice that $T_1=3$ and $T_2=2$.
	Thus, from Theorem~\ref{thm:Main-Results-x1}, 
	\begin{align*}
		\dim_B \Gamma f 
		= 1+ \max \Big\{0 , \frac{\log \uprho_1}{\log T_1}, \frac{\log \uprho_2}{\log T_2} \Big\}
		\approx  \max \Big\{1, \frac{\log 5.400}{\log 3}, \frac{\log 2.520}{\log 2} \Big\} 
		\approx  1.535.
	\end{align*}

\begin{table}[h]
\renewcommand\arraystretch{1.5}
\centering
\caption{spectral radii of $\upM_{1,k}|_{\Theta_{1,k}}$ and $\upM_{2,k}|_{\Theta_{2,k}}$ . }
\label{tab:example}
\begin{tabular}{|c|c|c|c|c|c|c|c|c|c|c|c|c|c|c|}
\hline
$k$ & 1 & 2 & 3 & 4 &5 & 6 & 7 & 10 \\ \hline
$\rho (\upM_{1,k}|_{\Theta_{1,k}})$ & 1.8793 & 1.8259 & 1.8085 & 1.8027 & 1.8008 & 1.8002 & 1.8000   & 1.7999 \\ \hline
$\rho (\lowM_{1,k}|_{\Theta_{1,k}})$ & 1.7180 & 1.7722 & 1.7906 & 1.7968 & 1.7988 & 1.7995 & 1.7998  & 1.7999 \\ \hline
$\rho (\upM_{2,k}|_{\Theta_{2,k}})$ & 1.2925 & 1.2758 & 1.2678 & 1.2638 & 1.2618 & 1.2608 & 1.2603  & 1.2598 \\ \hline
$\rho (\lowM_{2,k}|_{\Theta_{2,k}})$ & 1.2272 & 1.2433 & 1.2516 & 1.2557 & 1.2577 & 1.2587 & 1.2592 & 1.2597 \\ \hline
\end{tabular}
\end{table}

\bigskip
\noindent{\bf Acknowledgements.}
The research of Li was supported by NSFC grant 11801592. The research of Liang was supported by the Natural Science Foundation of Higher Education Institutions of Jiangsu Province (22KJB110018). The research of Ruan was supported in part by NSFC grants 12371089, 12531004, and the Fundamental Research Funds for the Central Universities of China grant 2024FZZX02-01-01.

\bibliographystyle{amsplain}

\end{document}